\documentclass[12pt,reqno]{amsart}
\usepackage{amsmath}
\usepackage{amssymb}
\usepackage{amstext}
\usepackage{mathrsfs}
\usepackage{a4wide}
\usepackage{graphicx}
\usepackage{bbm}
\usepackage{verbatim}
\usepackage{bm}
\usepackage{graphicx}
\usepackage{tikz}
\usepackage[colorlinks,linkcolor=blue]{hyperref}
\usepackage{pgfplots}
\usepackage{tikz}
\usepackage {graphicx} 
\usepackage {epstopdf}
\usepackage{float}

\allowdisplaybreaks \numberwithin{equation}{section}
\usepackage{color}
\usepackage{cases}

\numberwithin{equation}{section}

\newtheorem{theorem}{Theorem}[section]

\newtheorem{corollary}[theorem]{Corollary}
\newtheorem{lemma}[theorem]{Lemma}

\theoremstyle{definition}

\theoremstyle{remark}
\newtheorem{remark}[theorem]{Remark}

\newcommand{\ep}{\varepsilon}

\begin{document}

	\title
	[Steady vortex patches on flat torus]{Steady vortex patches on flat torus with a constant background vorticity} 
	
	\author{Takashi Sakajo, Changjun Zou}
	
	\address{Department of Mathematics, Kyoto University, Kyoto, 606--8502, Japan}
	\email{sakajo@math.kyoto-u.ac.jp}
	\address{Department of Mathematics, Sichuan University, Chengdu, Sichuan, 610064, P.R. China}
	\email{zouchangjun@amss.ac.cn}
	
	\thanks{}

	\begin{abstract}
		We construct a series of vortex patch solutions in a doubly-periodic rectangular domain (flat torus), which is accomplished by studying the contour dynamic equation for patch boundaries. We will illustrate our key idea by discussing the single-layered patches as the most fundamental configuration, and then investigate the general construction for $N$ patches near a point vortex equilibrium. Different with the case of bounded domains in $\mathbb R^2$, a constant background vorticity will arise from the compact nature of flat torus, and the $2$-dimensional translational invariance will bring troubles on determining patch locations. To overcome these two difficulties, we will add additional  terms for background vorticity and introduce a centralized condition for location vector. By utilizing the regularity difference of terms in contour dynamic equations, we also obtain the $C^\infty$ regularity and convexity of boundary curves.
		
	\end{abstract}
	
	\maketitle{\small{\bf Keywords:}  Incompressible Euler equation; Vortex patch problem; Constant background vorticity; Contour dynamic equation \\ 
		
		}
	
	\tableofcontents

	\section{Introduction}\label{sec1}

		The well-posedness of Euler equation is one of the central topics in fluid mechanics dating back to Leonhard Euler and his masterpiece “Principes g\'en\'eraux du mouvement des fluides”. Compared to the 3D case where several difficult problems remain open, the study for well-posedness of the incompressible 2D Euler equation has made considerable strides in the last century: 
		\begin{align} \label{Euler-eq}
			\begin{cases}
				\partial_t\omega+\boldsymbol{v}\cdot \nabla\omega =0&\text{in}\ \mathbb R^2,\\
				\boldsymbol{v}=\nabla^\perp(-\Delta)^{-1}\omega    &\text{in}\ \mathbb R^2,\\
				\omega\big|_{t=0}=\omega_0,
			\end{cases}
		\end{align}
		where $\boldsymbol v$ is the velocity field, $\omega$ is the vorticity scalar function, the first equation is the conservation of momentum, and the second line is Biot-Savart law with $(x_1,x_2)^\perp=(x_2,-x_1)$.
		For  smooth initial vorticity $\omega_0$, the global existence and uniqueness of smooth solutions of the Cauchy problem (\ref{Euler-eq})   went back to the work of H\"older \cite{Ho} and Wolibner \cite{Wo} in 1930s. 
		For initial vorticity only in $L^\infty$, the global existence and uniqueness of weak solution were proved by Yudovich \cite{Yud} in 1963. 
		Then DiPerna and Majda \cite{Dip} proved the existence of weak solutions for $\omega_0$ in $L^1\cap L^p$, and Delorit~\cite{Del} investigated a general situation where the vorticity is a signed measure in $H^{-1}$. 
		We refer to the book by Majda and Bertozzi \cite{MB} as a convenient introduction for the whole story.
		
		Besides the well-posedness of the Cauchy problem, another interesting topic is the construction for different kind of global solutions to the Euler equation, which maintain their shape during time evolution. 
		For the 3D Euler equation, the Hill spherical vortex \cite{Hil}, the vortex rings of small cross-section \cite{Fra1}, and the helical vortices \cite{Dav2} are such typical global solutions. 
		In the 2D situation, due to the structure of nonlinear term $\boldsymbol{v}\cdot \nabla\omega $, any radically symmetric function $\omega(|\boldsymbol x|)$  formulate a trivial steady solution in $\mathbb R^2$. 
		Among the nontrivial ones, the Lamb-Chaplygin diplole \cite{Lamb} and the Kirchhoff ellipse \cite{Kir} are two famous examples. 
		On the other hand, some special configurations of singular point vortices constitute the second kind of global solutions. 
		For example, the traveling vortex dipole and the rotating $k$-fold vortices are the most fundamental ones in $\mathbb R^2$.  
		

		One process constructing such relatively regular solutions is to approximate singular point vortex configuration, which  is called the regularization or desingularization of equilibrium states of point vortices. 
		A point vortex corresponds to a singular point-wise vorticity distribution with Dirac's $\delta$ measure at a point.
		In \cite{T}, Turkington applied a dual variational principle of Arnol'd \cite{Ar2} to regularize single steady point vortex by patch solutions in bounded and unbounded domain, which is accomplished by maximizing the kinetic energy in an arrangement class of vorticity function $\omega$. 
		This approach, also known as the vorticity method nowadays, is greatly developed by Cao et al.~\cite{Cao4,CQZZ2,CWZ} for the regularization with general vorticity distribution functions. 
		By letting $\psi=(-\Delta)^{-1}\omega$ be the Stokes stream function, another mainstream for the regularization is to consider the semi-linear elliptic equation on $\psi$, which is known as the stream function method. 
		Smets and Van Schaftingen~\cite{SV} applied this method together with the mountain pass variational techniques to obtain a family of $C^1$ localized vorticity solutions, and Cao et al.~\cite{Cao3} obtained the patch type regularization for an $N$-point vortex system by singular perturbation theory.
		We refer the interested readers to \cite{Ao,CQZZ3} for more information on this method.

		Besides the two methods stated above, a different approach, called the method of contour dynamics, is known. 
		It can be traced Burbea \cite{Burb} who studied the dynamics of vortex patch boundaries so as to obtain bifurcation curves of nontrivial solutions from the unit disc. For solutions on one curve, they are all of $k$-fold symmetry and rotating near a specific angular velocity, also known as the V-states.
		However, there is a gap for derivation on Hardy space in \cite{Burb}, which is mended by Hmidi et al. \cite{Hmi} in recent years using suitable H\"older spaces. 
		Since then, the method of contour dynamics has been widely used in the construction of V-states for active scalar equations.
		For instance, De La Hoz, Hassaina and Hmidi \cite{de2} proved the existence of doubly-connected vortex patches bifurcated from annuli at a specific angular velocity, and $2$-fold symmetric V-states near Kirchhoff elliptic vortices \cite{Kir} and V-states with smooth vorticity distribution were discovered by Castro, C\'ordoba and G\'omez-Serrano~\cite{Cas1,Cas2}.
		By using the global implicit theorem, Hassainia, Masmoudi, and Wheeler \cite{HMW} made continuation on the bifurcation parameter by an analytic tool, and studied the global behavior of solutions.
		Berti, Hassainia and Masmoudi~\cite{BHM} obtained qusi-periodic patches by Nash-Moser iteration. 
		There are also some parallel results for the generalized surface quasi-geostrophic (gSQG) equation, see \cite{Cas,Cas3,de1,Has}.
		
		Compared with plenty of results on the construction of V-states stated above, the method of contour dynamics is newly applied in the regularization of point vortices. 
		For instance, this method is used to construct traveling and co-rotating patches for gSQG equation (including 2D Euler equation) in the whole plane $\mathbb R^2$~\cite{HM,Cao1}. 
		It should be noticed that these two cases are only two special situations of regularization, since there is no influence of domain boundary nor background vorticity.   
		The purpose of the present is to establish the existence of a stationary vortex patch solution of the 2D Euler equation near a point-vortex equilibrium in a doubly connected domain.  
		We will also investigate the $C^\infty$ regularity and the convexity of patch boundary curves by further analysis, which cannot be derived by the standard vorticity and stream function method.
		
		Flows in the doubly periodic domains are often used in numerical calculations of two-dimensional homogeneous, isotropic turbulence. 
		It has been pointed out that the interaction of coherent vortex structures plays an important role in the understanding of the statistical laws of 2D turbulence \cite{Jimenez07,McWilliams84}. 
		Point vortices and vortex patches are the simple mathematical models of such interacting coherent vortex structures.
		Although a lattice configuration of a stationary vortex patch was obtained numerically by using the method of shape derivative \cite{Uda18}, the existence of stationary vortex patch lattice in the double periodic domain is not well established mathematically. 
		Tkachenko \cite{Tka66} considered the motion of a one-point vortex in a doubly periodic domain in a rotating frame of reference, and Benzi and Lagras \cite{BeLe87} discussed the interaction of point vortices embedded in a continuous background vorticity.
		The equation of $N$ point vortices in a doubly periodic domain was derived by O'Neil \cite{ONeil87}, which was used to investigate an equilibrium state of point vortices with some symmetry \cite{ONeil89}.
		In addition, under the condition that the sum of the point vortex circulations is zero, the equation of point-vortex dynamics have been derived in several forms, providing some stationary point vortex equilibria \cite{Crowdy10, StAr99, Weiss91}. 
		On the other hand, Krishnamurthy and  Sakajo \cite{VS-flat-torus} generalized the equations of $N$ point vortices in a rectangular doubly-periodic domain when the sum of the vortex strengths is non-zero, thereby finding many point vortex lattice equilibria. 
		These point vortex equilibria are used to construct stationary vortex patches here.
		
		Our paper is organized as follows. In Section \ref{sec:results}, we provide the mathematical formulation of the problem and state the main results. 
		In Section \ref{sec2}, we provide the Green function of $-\Delta$ in a doubly periodic domain by complex analysis and conformal mapping, and we then derive the contour dynamic equation for single-layered patches. 
		In Section \ref{sec3}, we establish the functional analysis frame for the construction in suitable Hilbert spaces on $\mathbb S^1$, and linearized the equation at point vortex equilibrium. 
		The main theorems are proved by a modified implicit functional theorem and bootstrap procedure separately in Section \ref{sec4} . 
		In Section \ref{sec5}, we generalize the method of construction, where a similar approach is applied on a vector form equation. 
		Some necessary properties of the $P$-function and its logarithmic derivative in Section \ref{sec2} are discussed in Appendix \ref{appA}.

		\section{Mathematical formulation and main results}\label{sec:results}
		We consider the steady 2D Euler equation in a doubly-periodic rectangular domain (flat torus)  $\mathbb{T}^2:(0,2\pi)\times (0,-\log\rho)$ with $0<\rho<1$, which reads
		\begin{align}\label{1-1}
			\begin{cases}
				\boldsymbol{v}\cdot \nabla\omega =0&\text{in}\ \mathbb{T}^2,\\
				\boldsymbol{v}=\nabla^\perp(-\Delta)^{-1}\omega    &\text{in}\ \mathbb{T}^2,\\
				\omega(x_1,0)=\omega(x_1,-\log\rho) &\text{for}\ x_1\in (0,2\pi),\\
				\omega(0,x_2)=\omega(2\pi,x_2) &\text{for}\ x_2\in (0,-\log\rho).
			\end{cases}
		\end{align}
		To elucidate our main strategy, we begin with the most fundamental single-layered point-vortex equilibrium embedded in a constant vorticity \cite{VS-flat-torus}.
		For $0<d<2\pi/N$, $0<h<-\log\rho$, let
		$$\boldsymbol P_n=\left(d_0+\frac{2\pi n}{N} \right), \quad n=0,1,\cdots,N-1,$$
		be $N$ points in $\mathbb{T}^2$ located at the row $x_2=h$. 
		It was shown that
		\begin{equation}\label{1-2}
			\omega^*(\boldsymbol x)=\pi\sum_{n=0}^{N-1}\boldsymbol\delta_{\boldsymbol P_n}-N\pi
		\end{equation}
		gives a steady configuration of point vortices in $\mathbb{T}^2$ by studying the Green function of $-\Delta$, where $\boldsymbol\delta_{\boldsymbol P_n}$ denotes the Delta measure at $\boldsymbol P_n$.
		Thus we shall construct steady patch solutions to \eqref{1-1} taking the form 
		\begin{equation}\label{1-3}
			\omega(\boldsymbol x)=\frac{1}{\ep^2}\sum_{n=0}^{N-1}\boldsymbol 1_{\Omega_n^\ep}-\gamma_\ep,
		\end{equation}
		where $\boldsymbol 1_\Omega$ denotes the characteristic function of domain $\Omega$, 
		$$\Omega_n^\ep=\Omega_0^\ep+\frac{2\pi n}{N}\boldsymbol e_1, \quad n=0,1,\cdots,N-1$$ 
		be $N$ small areas as perturbation of the circle with radius $\ep>0$ around $\boldsymbol P_n$, say  $B_\ep(\boldsymbol P_n)$, and $\gamma_\ep\in \mathbb{R}$ the strength of a constant background vorticity near $N\pi$. 
		Actually, we will use $\gamma_\ep$ as the bifurcation parameter in the functional analysis frame of the construction, which ensures that the contour dynamic equation induces an isomorphism between well-selected Hilbert spaces. 
		Then by applying the implicit function theorem near at the singular point vortex equilibrium $\omega^*$ in \eqref{1-2}, we have following result on the existence of single-layered vortex patches.
		\begin{theorem}\label{thm1}
			There exists $\ep_0>0$ such that for any $\varepsilon\in [0,\ep_0)$, \eqref{1-1} has a steady patch solution \eqref{1-3}, where the strength of background vorticity $\gamma_\ep$ satisfies
			$$\gamma_\ep=N\pi+O(\varepsilon).$$
		\end{theorem}
		
		\begin{remark}
			In the above theorem, we use the single-layered configuration as an example to show the existence of patch solutions. 
			However, due to the doubly-periodicity of $\mathbb{T}^2$, any $M$-layered patch solutions with $M\ge1$ can be constructed in the same way, as long as the distance of two neighboring layers is $h/M$. 
			These $M$-layered patches constitute the first class of patch solutions with nonzero ground vorticity, of which only a few is known nowadays near corresponding point vortex equilibria, see \cite{VS-flat-torus}. 
			On the other hand, solutions with zero background vorticity can also be constructed by our approach, where the positive and negative patches have equivalent circulations. 
			For this case, the proof is much more easier compared with that for Theorem \ref{thm1} since the linearized operator of contour dynamic equation will formulate an isomorphism itself.
		\end{remark}
		
		\begin{remark}
			For simplicity, in the construction for single-layered vortex patches, we assume that all $\Omega_n^\ep$, $n=0,\cdots,N-1$ are generated by translation and of a same shape, and even-symmetric with $x_2=h$.
			The first assumption simplifies the whole system to only one contour dynamic equation, and the second ensures that the parametrization of $\partial\Omega_n^\ep$ is only expanded by those even $\cos(js)$.
			In fact, these two constraints are unnecessary in the construction for general patch configurations, which will be showed in the proof of Theorem \ref{thm3}.
		\end{remark}  
		
		Besides the existence of vortex patches, another interesting topic is the regularity and convexity of the patch boundaries $\partial \Omega_n^\ep$, $n=0,\cdots, N-1$. For V-state solutions to gSQG equation, the $C^\infty$ regularity of patch boundary is proved in \cite{Cas,Hmi} via complex analysis and bootstrap procedure. 
		A similar method is applied in \cite{Cao1}, where the boundary smoothness and convexity for traveling and co-rotating patches for gSQG equation with $0<s\le1/2$ in $(-\Delta)^s$ are proved. 
		However, as far as we know, there is no such result on the regularization of 2D Euler equation serving as the most classic situation. 
		To give an affirmative answer to this question, we will prove the following theorem based on the the regularity  of the parch boundaries.
		\begin{theorem}\label{thm2}
			For the steady patch solution \eqref{1-2} obtained in Theorem \ref{thm1}, the patch area $\Omega_0^\ep$ is a convex domain, whose boundary can be parameterized by a smooth function $R(s)\in C^\infty(\mathbb S^1)$.
		\end{theorem}
		
		\begin{remark}
			In Theorem \ref{thm2}, we obtain the $C^\infty$ regularity and the convexity of patch boundary, which is an advantage of considering contour dynamic equation for patch boundaries. 
			In \cite{T}, by the energy comparison and the earrangement inequalities, Turkington claims that unit disk is the limit function for each patch as $\ep\to 0^+$, but his variational approach cannot describe the regularity of patch boundary for each $\ep>0$. 
			While a perturbation method is applied near a scaled Rankine vortex to obtain the $C^1$ regularity of patch boundary in \cite{Cao1}, and it is the best result by stream function method limited by the regularity theory for elliptic operators.  
			However, to derive the convexity of $\Omega_0^\ep$, at least a $C^2$ estimate is required for the calculation of curvature, which is achieved in the present paper by a bootstrap procedure for the contour dynamic equation. 
		\end{remark}
		
		\smallskip
		
		Following the idea in the proof of Theorem \ref{thm1}, we can further construct patch solutions near a general steady point vortex configuration. Unlike the single-layered case where we assume that all patch are even-symmetric with a same shape so that we can disturb $\gamma_\ep$ only, for the general $N$-patch configuration we need to adjust $2N$ parameters. Thus these $N$ patches are assumed to be located near an $N$-point vortex equilibrium, and we will determine $2N$ center location parameters in the last step.
		According to formula \eqref{2-4} in Section \ref{sec2}, the Green function for $-\Delta$ in $\mathbb{T}^2$ can be expanded as 
		\begin{equation*}
			G(\boldsymbol x,\boldsymbol y)=\frac{1}{2\pi}\log \frac{1}{|\boldsymbol x-\boldsymbol y|}+H(\boldsymbol x,\boldsymbol y)-\frac{1}{4\pi\log\rho}|\boldsymbol x-\boldsymbol y|^2,
		\end{equation*}
		which is also the stream function for a single point vortex located at $\boldsymbol y\in \mathbb{T}^2$ with the unit circulation. 
		Using this fact, we can write down the Kirchhoff-Routh path function as the Hamiltonian for an $N$-point vortex system with circulation parameters $(\kappa_1,\cdots,\kappa_N)\in \mathbb R^N$:
		\begin{equation}\label{1-4}
			\mathcal W_N(\boldsymbol x_1,\cdots,\boldsymbol x_N)=\sum_{n,m=1,n\le m}^N\kappa_m\kappa_nG(\boldsymbol x_m,\boldsymbol x_n)+\frac{1}{2}\sum_{m=1}^N \kappa_m^2H(\boldsymbol x_m,\boldsymbol x_m).
		\end{equation}
		According to \cite{Lin}, every critical point $\boldsymbol X^0=(\boldsymbol x^0_1,\cdots,\boldsymbol x^0_N)\in \mathbb{T}^N$ of $\mathcal W_N$ corresponds to an equilibrium state of point vortex
		\begin{equation}\label{1-5}
			\omega^*(\boldsymbol x)=\sum_{n=1}^N\kappa_n\boldsymbol\delta_{\boldsymbol x^0_n}-\sum_{n=1}^N\kappa_n.
		\end{equation}
		Since the flat torus $\mathbb{T}^2$ has two-dimensional translational invariance, the critical points of \eqref{1-4} are at least two-dimensional degenerate, namely, $\mathrm{Rank}\left(\mathrm{Hess}(\mathcal W_N(\boldsymbol X^0))\right)=2N-2$. 
		This phenomenon is different with general bounded subdomains in $\mathbb R^2$, where in most situations the critical point can be non-degenerate.
		
		Thus, to fix the locations of vortex patches, we give the following definition: A vector $\boldsymbol X=(\boldsymbol x_1,\cdots,\boldsymbol x_N)\in  (\mathbb{T}^2)^N$ is said to be \emph{centralized} if the center of $\boldsymbol X$ is at $(\pi,-\frac{\log\rho}{2})$, namely,
		\begin{equation*}
			\sum\limits_{n=1}^Nx_{n1}=N\pi, \quad\quad \mathrm{and} \quad\quad \sum\limits_{n=1}^Nx_{n2}=-\frac{N\log\rho}{2},
		\end{equation*}
		where $\boldsymbol x_n=(x_{n1}, x_{n2})$.
		By assuming that $\boldsymbol X^0=(\boldsymbol x^0_1,\cdots,\boldsymbol x^0_N)$ is a two-dimensional degenerate centralized critical point of $\mathcal W_N(\boldsymbol x_1,\cdots,\boldsymbol x_N)$ condition, we will construct a family of patch solutions taking the form
		\begin{equation}\label{1-6}
			\omega(\boldsymbol x)=\frac{1}{\ep^2}\sum_{n=1}^{N}\frac{\kappa_n}{\pi}\boldsymbol 1_{\Omega_n^\ep}-\sum_{n=1}^N\kappa_n,
		\end{equation}
		where $\Omega_1^\ep,\cdots,\Omega_N^\ep$ are $N$ small areas as the perturbation of disks $B_\ep(\boldsymbol x_1^\ep),\cdots,B_\ep(\boldsymbol x_N^\ep)$ (the radius of these $N$ disks can be different, but we unify them to $\ep$ for simplicity), and the centralized location vector $\boldsymbol X^\ep=(\boldsymbol x^\ep_1,\cdots,\boldsymbol x^\ep_N)$ is an $O(\ep)$-perturbation of $\boldsymbol X^0=(\boldsymbol x^0_1,\cdots,\boldsymbol x^0_N)$. 
		By studying the  contour dynamic equations on the $N$ boundaries of $\Omega_n^\ep$, $n=1,\cdots,N$, the general existence theorem is obtained as follows.
		\begin{theorem}\label{thm3}
			Fix the circulation parameters $(\kappa_1,\cdots,\kappa_N)\in \mathbb R^N$, if $\boldsymbol X^0\in (\boldsymbol x_1,\cdots,\boldsymbol x_N)\in (\mathbb{T}^2)^N$ is a centralized critical point of $\mathcal W_N$ in \eqref{1-4} such that 
			$$\mathrm{Rank}\,\left(\mathrm{Hess}(\mathcal W_N(\boldsymbol X^0))\right)=2N-2,$$
			then there exists $\epsilon_0>0$ such that for any $\varepsilon\in [0,\epsilon_0)$, \eqref{1-1} has a steady patch solution \eqref{1-6}, where the patch location vector $\boldsymbol X^\ep\in (\boldsymbol x_1^\ep,\cdots,\boldsymbol x_N^\ep)\in (\mathbb{T}^2)^N$ is centralized, and satisfies
			$$\boldsymbol X^\ep=\boldsymbol X^0+O(\ep)$$
			Moreover, each patch area $\Omega_m^\ep$ is a convex domain, whose boundary can be parameterized by a smooth function $R_m(s)\in C^\infty(\mathbb S^1)$.
		\end{theorem}
		
		\begin{remark}
			Since the shapes of $\Omega_n^\ep$, $n=1,\cdots,N$ are diverse and no symmetry is assumed, the construction for general situation stated in Theorem \ref{thm3} will be transformed into solving an integro-differential system with $N$ equations, and the solution $R_m(s)$ is composed of both $\cos(js)$ and $\sin(js)$. 
			Thus, instead of letting the background vorticity $\gamma_\ep$ change as in the proof of Theorem \ref{thm1}, here we will adjust the location vector $\boldsymbol X^\ep$ so that the linearized operator will formulate an isomorphism.
			Actually, this condition is equivalent to  an algebraic system for $2N$ variables, which can be solved using the two-dimensional degenerate property of $\mathcal W_N(\boldsymbol X^0)$ together with the two-dimensional centralized assumption.
		\end{remark}    
		
		\begin{remark}
			By defining the Kirchhoff-Routh function $\mathcal W_N$ in \eqref{1-4}, we assume that the background vorticity is contributed by $-\sum_{n=1}^N\frac{\kappa_n}{4\pi\log\rho}|\boldsymbol x-\boldsymbol x_n|^2$ from all $N$ vortices. 
			This is not the only choice but the most standard one, especially when considering the dynamic of vortices. 
			For other types of background vorticity, one can also apply our method of construction by finding out the critical point of corresponding the Kirchhoff-Routh function.
		\end{remark}    
		
		\smallskip

		\bigskip
		
	   \section{The Green function and contour dynamics}\label{sec2}
	   
	   In this section, we will obtain the precise expansion for the Green function $G$ by complex analysis, and derive the equation for the patch boundary.
	   
	   \subsection{Expansion for the Green function $G$}
	   
	   The Green function $G(\boldsymbol x, \boldsymbol y)$ for $\mathbb{T}^2$ will satisfy
	   \begin{equation}\label{2-1}
	   	-\Delta_{\boldsymbol x}G(\boldsymbol x,\boldsymbol y)=\boldsymbol\delta_{\boldsymbol y}-\frac{1}{|D|}
	   \end{equation} 
	   together with the doubly-periodic boundary conditions in \eqref{1-1}. 
	   Note that the term $1/|D|$ is required by the Gauss divergence theorem for the compact nature of the domain, and $|D|=-2\pi\ln\rho>0$ is the area of $\mathbb{T}^2$. 
	   
	   To derive a precise expansion for $G(\boldsymbol x, \boldsymbol y)$ in \eqref{2-1}, instead of dealing with domain $D\subset\mathbb R^2$ directly, we consider a doubly-periodic rectangle $\mathcal D_{\boldsymbol z}\subset \mathbb C$ in the complex $\boldsymbol z$-plane with the lattice defined by the fundamental pair of periods $2\pi$ and $-\log\rho$, where the transform is given by $$\boldsymbol x=(x_1,x_2)\in \mathbb{T}^2 \to \boldsymbol z=x_1+\mathbf i x_2\in \mathcal D_{\boldsymbol z}.$$
	   In these complex settings, the planar Laplacian $-\Delta$ can be denoted by the Wirtinger derivative $-4\partial^2_{\boldsymbol z\bar{\boldsymbol z}}$, and the hydrodynamic Green function $\mathcal G(\boldsymbol z,\boldsymbol w;\bar{\boldsymbol z},\bar{\boldsymbol w})$ on $\mathcal D_{\boldsymbol z}$ can be defined as the unique real-valued function for $\boldsymbol z,\boldsymbol w\in \mathbb C$, such that the following conditions hold.
	   
	   \begin{itemize}
	   	\item[(i)] $\mathcal G(\boldsymbol z,\boldsymbol w;\bar{\boldsymbol z},\bar{\boldsymbol w})$ has a logarithmic singularity at $\boldsymbol z=\boldsymbol w$ and satisfies the equation
	   	$$-4\partial^2_{\boldsymbol z\bar{\boldsymbol z}}\mathcal G=\boldsymbol \delta_{\boldsymbol w}-\frac{1}{|\mathcal D_{\boldsymbol z}|},$$
	   	which implies the existence of a function 
	   	$$\mathcal H(\boldsymbol z,\boldsymbol w;\bar{\boldsymbol z},\bar{\boldsymbol w}):=\mathcal G(\boldsymbol z,\boldsymbol w;\bar{\boldsymbol z},\bar{\boldsymbol w})-\frac{1}{2\pi}\log\frac{1}{|\boldsymbol z-\boldsymbol w|}.$$
	   	It is regular in  $\mathcal D_{\boldsymbol z}$ satisfying
	   	$$-4\partial^2_{\boldsymbol z\bar{\boldsymbol z}}\mathcal H=\frac{1}{2\pi\log\rho}.$$
	   	\item[(ii)] $\mathcal G(\boldsymbol z,\boldsymbol w;\bar{\boldsymbol z},\bar{\boldsymbol w})$ is doubly-periodic in both arguments:
	   	$$\mathcal G(\boldsymbol z+2m\pi-\mathbf{i}n\log\rho,\boldsymbol w;\bar{\boldsymbol z}+2m\pi-\mathbf{i}n\log\rho,\bar{\boldsymbol w})=\mathcal G(\boldsymbol z,\boldsymbol w;\bar{\boldsymbol z},\bar{\boldsymbol w}),$$
	   	$$\mathcal G(\boldsymbol z,\boldsymbol w+2m\pi-\mathbf{i}n\log\rho;\bar{\boldsymbol z},\bar{\boldsymbol w}+2m\pi-\mathbf{i}n\log\rho)=\mathcal G(\boldsymbol z,\boldsymbol w;\bar{\boldsymbol z},\bar{\boldsymbol w})$$
	   	for all $m,n\in\mathbb Z$.
	   	\item[(iii)] $\mathcal G(\boldsymbol z,\boldsymbol w;\bar{\boldsymbol z},\bar{\boldsymbol w})$ is symmetric with respect to $\boldsymbol z$ and $\boldsymbol w$:
	   	$$\mathcal G(\boldsymbol z,\boldsymbol w;\bar{\boldsymbol z},\bar{\boldsymbol w})=\mathcal G(\boldsymbol w,\boldsymbol z;\bar{\boldsymbol w},\bar{\boldsymbol z}).$$
	   \end{itemize}	
	   
	   \smallskip
	   
	   To obtain the exact formula for $\mathcal G(\boldsymbol z,\boldsymbol w;\bar{\boldsymbol z},\bar{\boldsymbol w})$, we introduce  a conformal map
	   \begin{equation}\label{2-2}
	   	\boldsymbol z(\boldsymbol \zeta)=-\mathbf i\log\boldsymbol\zeta
	   \end{equation}
	   from the complex $\boldsymbol \zeta$-plane to the complex $\boldsymbol z$-plane. 
	   Then the pre-image of $\mathcal D_{\boldsymbol z}$ is the annulus
	   \begin{equation}\label{annular-domain}
	   	\mathcal D_{\boldsymbol \zeta}=\{\boldsymbol\zeta\in\mathbb C \mid \rho<|\boldsymbol \zeta|\le 1, 0\le\arg \boldsymbol z<2\pi \}.
	   \end{equation}
	   The Laplacian in $\mathcal D_{\boldsymbol z}$ transforms to
	   $$-4\partial^2_{\boldsymbol z\bar{\boldsymbol z}}=\frac{4}{|\boldsymbol z'(\boldsymbol\zeta)|^2}\partial^2_{\boldsymbol \zeta\bar{\boldsymbol \zeta}}=-4|\boldsymbol\zeta|^2\partial^2_{\boldsymbol \zeta\bar{\boldsymbol \zeta}}.$$
	   We denote the hydrodynamic Green function in the annulus by 
	   $\mathcal G^*(\boldsymbol \zeta,\boldsymbol \nu;\bar{\boldsymbol \zeta},\bar{\boldsymbol \nu})=\mathcal G(\boldsymbol z,\boldsymbol w;\bar{\boldsymbol z},\bar{\boldsymbol w}),$
	   where $\boldsymbol \zeta,\boldsymbol\nu\in \mathcal D_{\boldsymbol \zeta}$ are the pre-images of $\boldsymbol z,\boldsymbol w\in\mathcal D_{\boldsymbol z}$ respectively. 
	   According to conditions (i)-(iii), this Green function $\mathcal G^*(\boldsymbol \zeta,\boldsymbol \nu;\bar{\boldsymbol \zeta},\bar{\boldsymbol \nu})$ will satisfy 
	   $$-4|\boldsymbol\zeta|^2\partial^2_{\boldsymbol z\bar{\boldsymbol z}}\mathcal G^*=\boldsymbol \delta_{\boldsymbol \nu}-\frac{1}{2\pi\log\rho},$$
	   $$\mathcal G^*(\rho^n\boldsymbol \zeta,\rho^m\boldsymbol \nu;\rho^n\bar{\boldsymbol \zeta},\rho^m\bar{\boldsymbol \nu})=\mathcal G^*(\boldsymbol \zeta,\boldsymbol \nu;\bar{\boldsymbol \zeta},\bar{\boldsymbol \nu}), \quad \forall n, m\in\mathbb Z,$$
	   and
	   $$\mathcal G^*(\boldsymbol \zeta,\boldsymbol \nu;\bar{\boldsymbol \zeta},\bar{\boldsymbol \nu})=\mathcal G^*(\boldsymbol \nu,\boldsymbol \zeta;\bar{\boldsymbol \nu},\bar{\boldsymbol \zeta}).$$
	   Using the $P$-function associated with the annulus (\ref{annular-domain}), $\mathcal G^*(\boldsymbol \zeta,\boldsymbol \nu;\bar{\boldsymbol \zeta},\bar{\boldsymbol \nu})$ is written down directly as follows~\cite{VS-flat-torus}.
	   \begin{equation}\label{2-3}
	   	\mathcal G^*(\boldsymbol \zeta,\boldsymbol \nu;\bar{\boldsymbol \zeta},\bar{\boldsymbol \nu})=-\frac{1}{2\pi}\log|P(\boldsymbol \zeta/\boldsymbol \nu,\sqrt\rho)|+\frac{1}{4\pi}\log|\boldsymbol \zeta/\boldsymbol \nu|-\frac{1}{4\pi\log\rho}(\log|\boldsymbol \zeta/\boldsymbol \nu|)^2.
	   \end{equation}
	   The basic properties of $P(\cdot,\sqrt\rho)$ and its log-derivative are collected in Appendix \ref{appA}, and one can verify
	   $$-4|\boldsymbol\zeta|^2\partial^2_{\boldsymbol \zeta\bar{\boldsymbol \zeta}}\left(-\frac{1}{4\pi\log\rho}(\log|\boldsymbol \zeta/\boldsymbol \nu|)^2\right)=\frac{1}{2\pi\log\rho}$$
	   by simple computation.
	   
	   By \eqref{2-2} and \eqref{2-3}, we have
	   \begin{equation*}
	   	\mathcal G(\boldsymbol z,\boldsymbol w;\bar{\boldsymbol z},\bar{\boldsymbol w})=-\frac{1}{2\pi}\log|P(e^{\mathbf i\boldsymbol z}/e^{\mathbf i\boldsymbol w},\sqrt\rho)|+\frac{1}{4\pi}|\boldsymbol z-\boldsymbol w|-\frac{1}{4\pi\log\rho}|\boldsymbol z-\boldsymbol w|^2.
	   \end{equation*}
	   Then, by identifying $\boldsymbol z=x_1+\mathbf i x_2, \boldsymbol w=y_1+\mathbf i y_2$ with $\boldsymbol x=(x_1,x_2), \boldsymbol y=(y_1,y_2)$ in the real coordinates, we show that the real Green function $G(\boldsymbol x,\boldsymbol y): \mathbb{T}^2\times \mathbb{T}^2\to \mathbb R$ satisfies
	   \begin{align*}
	   	G(\boldsymbol x,\boldsymbol y)&=\mathcal G(\boldsymbol z,\boldsymbol w;\bar{\boldsymbol z},\bar{\boldsymbol w})\\
	   	&=\mathcal G(x_1+\mathbf ix_2,y_1+\mathbf iy_2;x_1-\mathbf ix_2,y_1+\mathbf iy_2).
	   \end{align*}
	   Hence we can split $G(\boldsymbol x,\boldsymbol y)$ into three parts as
	   \begin{equation}\label{2-4}
	   	G(\boldsymbol x,\boldsymbol y)=\frac{1}{2\pi}\log \frac{1}{|\boldsymbol x-\boldsymbol y|}+H(\boldsymbol x,\boldsymbol y)-\frac{1}{4\pi\log\rho}|\boldsymbol x-\boldsymbol y|^2,
	   \end{equation}
	   where $H(\boldsymbol x,\boldsymbol y)\in C^\infty(\mathbb{T}^2\times \mathbb{T}^2)$, and $-\frac{|\boldsymbol x-\boldsymbol y|^2}{4\pi\log\rho}$ corresponds to the area term $1/|D|$ in \eqref{2-1}.
	   
	   \smallskip
	   
	   \subsection{Equilibrium states for point vortices in $\mathbb{T}^2$}
	   
	   If $v_1-\mathbf iv_2$ is the complex velocity for an incompressible flow in the $\boldsymbol z$-plane, then a stream function $\psi(\boldsymbol z,\bar{\boldsymbol z})$ exists and the velocity field $\boldsymbol v=v_1+\mathbf i v_2$ is defined via
	   \begin{equation}\label{2-5}
	   	\bar{\boldsymbol v}=v_1-\mathbf i v_2=2\mathbf i\frac{\partial \psi}{\partial \boldsymbol z}.
	   \end{equation}
	   The vorticity function $\omega(\boldsymbol z,\bar{\boldsymbol z})$ and $\psi(\boldsymbol z,\bar{\boldsymbol z})$ are related by
	   $$\omega=-\Delta\psi=-4\partial^2_{\boldsymbol z\bar{\boldsymbol z}}\psi,$$
	   and in the $\boldsymbol \zeta$-plane, \eqref{2-5} transforms into 
	   \begin{equation}\label{2-6}
	   	v_1-\mathbf i v_2=-2\boldsymbol\zeta\frac{\partial \psi}{\partial \boldsymbol \zeta}. 
	   \end{equation}

	   Since the Green function $\mathcal G(\boldsymbol z,\boldsymbol w;\bar{\boldsymbol z},\bar{\boldsymbol w})$ is actually the stream function for a point vortex located at $\boldsymbol w$ with the unit circulation, we can establish the dynamics of point vortex system in $\mathcal D_{\boldsymbol z}$ by summing the influence of all vortices together.
	   To eliminate the singular log-rotating term, we introduce the regular part of $\mathcal G$ as
	   \begin{equation}\label{2-7}
	   	\mathcal H(\boldsymbol z,\boldsymbol w;\bar{\boldsymbol z},\bar{\boldsymbol w})=\mathcal G(\boldsymbol z,\boldsymbol w;\bar{\boldsymbol z},\bar{\boldsymbol w})-\frac{1}{2\pi}\log\frac{1}{|\boldsymbol z-\boldsymbol w|}.
	   \end{equation}
	   Then for $N$ point vortices located at $\boldsymbol z_1,\cdots,\boldsymbol z_N$ with the circulations $\kappa_1,\cdots,\kappa_N$, by assuming that velocity field of the $j^{\mathrm{th}}$ vortex is $\boldsymbol v_j=v_{j1}+\mathbf i v_{j2}$, the equation of motion for the $j^{\mathrm{th}}$ vortex is given by
	   \begin{equation}\label{2-8}
	   	\frac{d \bar{\boldsymbol z}_m}{d t}=\bar{\boldsymbol v}_m=-2\boldsymbol\zeta\frac{\partial}{\partial\boldsymbol\zeta}\left[\sum_{n=1,n\neq m}^N\kappa_n\mathcal G^*(\boldsymbol \zeta,\boldsymbol \nu_n;\bar{\boldsymbol \zeta},\bar{\boldsymbol \nu}_n)+\kappa_m\mathcal H^*(\boldsymbol \zeta,\boldsymbol \nu_m;\bar{\boldsymbol \zeta},\bar{\boldsymbol \nu}_m)\right]\bigg|_{\boldsymbol \zeta=\boldsymbol \nu_m}
	   \end{equation}
	   with $\boldsymbol \nu_n$ being the pre-image of $\boldsymbol z_n$, and $\mathcal H^*(\boldsymbol \zeta,\boldsymbol \nu;\bar{\boldsymbol \zeta},\bar{\boldsymbol \nu})=\mathcal H(\boldsymbol z,\boldsymbol w;\bar{\boldsymbol z},\bar{\boldsymbol w})$ the regular part of the Green function in the $\boldsymbol \zeta$-plane.
	   Actually, \eqref{2-8} means that a point vortex system in $\mathcal D_{\boldsymbol z}$ admits a Hamiltonian structure 
	   $$\kappa_m\frac{d \bar{\boldsymbol z}_m}{d t}=-2\mathbf i\frac{\partial\mathcal W}{\partial \boldsymbol z_m},$$
	   where the Hamiltonian $\mathcal W^*_N(\boldsymbol z_1,\cdots, \boldsymbol z_N,\bar{\boldsymbol z}_1,\cdots, \bar{\boldsymbol z}_N)$ is given by
	   \begin{equation*}
	   	\mathcal W^*_N(\boldsymbol z_1,\cdots, \boldsymbol z_N,\bar{\boldsymbol z}_1,\cdots, \bar{\boldsymbol z}_N)=\sum_{n,m=1,n<m}^N\kappa_m\kappa_n\mathcal G^*(\boldsymbol z_m,\boldsymbol z_n;\bar{\boldsymbol z}_m,\bar{\boldsymbol z}_n)+\frac{1}{2}\sum_{m=1}^N\kappa_m^2\mathcal H^*(\boldsymbol z_m,\boldsymbol z_m;\bar{\boldsymbol z}_m,\bar{\boldsymbol z}_m).
	   \end{equation*}
	   This is also known as the Kirchhoff-Routh path function. See \cite{Lin}.
	   
	   In \eqref{2-7}, the regular function of one-variable $\mathcal H^*(\boldsymbol \zeta,\boldsymbol \zeta;\bar{\boldsymbol \zeta},\bar{\boldsymbol \zeta})$ is called the Robin function. Since 
	   $$\log|\boldsymbol z-\boldsymbol w|=\log|\boldsymbol z'(\boldsymbol\nu)(\boldsymbol\zeta-\boldsymbol\nu)|+O(|\boldsymbol\zeta-\boldsymbol\nu|)=\ln|\boldsymbol\zeta/\boldsymbol w-1|+O(|\boldsymbol\zeta-\boldsymbol\nu|),$$
	   by \eqref{2-3} together with the definition of the $P$-function \eqref{A-1} in Appendix \ref{appA}, we have
	   \begin{align*}
	   	\mathcal H^*(\boldsymbol \zeta,\boldsymbol \nu;\bar{\boldsymbol \zeta},\bar{\boldsymbol \nu})&=-\frac{1}{2\pi}\log\left|\prod_{n=1}^\infty(1-\rho^n\boldsymbol \zeta/\boldsymbol\nu)(1-\rho^n\boldsymbol \nu/\boldsymbol\zeta)\right|+\frac{1}{4\pi}\log|\boldsymbol\zeta/\boldsymbol\nu|\\
	   	&\quad-\frac{1}{4\pi\rho}(\log|\boldsymbol\zeta/\boldsymbol\nu|)^2+O(|\boldsymbol \zeta-\boldsymbol \nu|).
	   \end{align*}
	   Thus
	   $$\mathcal H^*(\boldsymbol \zeta,\boldsymbol \zeta;\bar{\boldsymbol \zeta},\bar{\boldsymbol \zeta})=-\frac{1}{2\pi}\left|\prod_{n=1}^\infty(1-\rho^k)^2\right|$$
	   is a constant. 
	   Consequently, from \eqref{2-8} we obtain
	   \begin{equation}\label{2-9}
	   	\begin{split}
	   		\frac{d \bar{\boldsymbol z}_m}{d t}=\bar{\boldsymbol v}_m&=-2\sum_{n=1,n\neq m}^N\kappa_n\boldsymbol\nu_m\frac{\partial}{\partial\boldsymbol\nu_m}\mathcal G^*(\boldsymbol \nu_m,\boldsymbol \nu_n;\bar{\boldsymbol \nu}_m,\bar{\boldsymbol \nu}_n)\\
	   		&=\frac{1}{2\pi}\sum_{n=1,n\neq m}^N\kappa_n K(\boldsymbol \nu_m/\boldsymbol \nu_n,\sqrt\rho)-\frac{1}{4\pi}\sum_{n=1,n\neq m}^N\kappa_n\\
	   		&\quad+\frac{1}{2\pi\rho}\sum_{n=1,n\neq m}^N\kappa_n\ln|\boldsymbol\nu_m/\boldsymbol\nu_n|, \quad \mathrm{for} \ m=1,\cdots,N,
	   	\end{split}
	   \end{equation}
	   where the $K$-function $K(\cdot, \sqrt\rho)$ is the logarithmic derivative of the $P$-function $P(\cdot, \sqrt\rho)$.
	   
	   Thanks to \eqref{2-9}, for $\boldsymbol x_n=(x_{n1},x_{n2})\in \mathbb{T}^2$, $\kappa_n\in\mathbb R$, $n=1,\cdots, N$, now we can state a sufficient and necessary condition for point vortex system 
	   \begin{equation*}
	   	\omega^*(\boldsymbol x)=\sum_{n=1}^N\kappa_n\boldsymbol \delta_{\boldsymbol x_n}-\sum_{n=1}^N\kappa_n
	   \end{equation*}
	   constituting an equilibrium state to \eqref{1-1}: For $m=1,\cdots, N$, if the location coordinates 
	   $$(\boldsymbol \nu_1,\cdots,\boldsymbol \nu_N)=(e^{\mathbf i(x_{11}+x_{12}\mathbf i)},\cdots e^{\mathbf i (x_{11}+x_{12}\mathbf i)})$$
	   and the circulation parameters $(\kappa_1,\cdots,\kappa_N)$ satisfy $N$ complex equations
	   \begin{equation}\label{2-10}
	   	\begin{split}
	   		0&=f_m(\boldsymbol \nu_1,\cdots,\boldsymbol \nu_N,\kappa_1,\cdots,\kappa_N)\\
	   		&=\frac{1}{2\pi}\sum_{n=1,n\neq m}^N\kappa_n K(\boldsymbol \nu_m/\boldsymbol \nu_n,\sqrt\rho)-\frac{1}{4\pi}\sum_{n=1,n\neq m}^N\kappa_n+\frac{1}{2\pi\rho}\sum_{n=1,n\neq m}^N\kappa_n\ln|\boldsymbol\nu_m/\boldsymbol\nu_n|,
	   	\end{split}
	   \end{equation}  
	   then $\omega^*$ is a steady solution to \eqref{1-1}. This condition is also equivalent to that $\boldsymbol x_n=(x_{n1},x_{n2})\in \mathbb{T}^2$ is a critical point of the Kirchhoff-Routh function $\mathcal W_N(\boldsymbol x)$ defined \eqref{1-4}.
	   
	   As a simplest example, we can verify an equilibrium state for point vortex solution to \eqref{1-1}. Consider $N$ vortices of identical strength $+1$ arranged on a polygon in $\mathcal D_{\boldsymbol \zeta}$, whose coordinates are given by the complex numbers:
	   $$\boldsymbol \nu_n=r e^{2\pi\mathbf i(n-1)/N}.$$
	   Using the properties \eqref{A-2}, \eqref{A-3}, and \eqref{A-4} in Appendix \ref{appA} for the $K$-function, if $N=2M$ with $M\in \mathbb N^*$, we can calculate to obtain
	   \begin{align}\label{2-11}
	   	&\sum_{n=2}^M\left[K(r/\boldsymbol \nu_n,\sqrt\rho)+K(r/\boldsymbol \nu_n,\sqrt\rho)\right]+K(-1,\sqrt\rho)-\frac{2M-1}{2}\\
	   	&=(M-1)+\frac{1}{2}-\frac{2M-1}{2}=0,\nonumber
	   \end{align}
	   and if $N=2M-1$ is odd, it will hold
	   \begin{align}\label{2-12}
	   	&\sum_{n=2}^{M+1}\left[K(r/\boldsymbol \nu_n,\sqrt\rho)+K(r/\boldsymbol \nu_n,\sqrt\rho)\right]-(M-1)\\
	   	&=(M-1)-(M-1)=0.\nonumber
	   \end{align}
	   The last step is to pull back the coordinates $\boldsymbol\nu_n$ from $\mathcal D_{\boldsymbol \zeta}$ to $\boldsymbol P_n$ in $D\subset \mathbb R^2$. Let $0<d<2\pi/N$, $0<h<-\ln\rho$ be two constants, and 
	   $$\boldsymbol P_k=\left(d+\frac{2\pi n}{N}, h\right), \quad n=0,1,\cdots,N-1,$$
	   be $N$ points located at the row $x_2=h$. Since 
	   $$-\mathbf i\log\boldsymbol \nu_n=\frac{(n-1)2\pi}{N}-\mathbf i\log r,$$ 
	   by \eqref{2-10}, \eqref{2-11} and doubly-periodicity of $\mathbb{T}^2$, we see that the point vortex system
	   $$\omega^*(\boldsymbol x)=\sum_{n=0}^{N-1}\boldsymbol \delta_{\boldsymbol P_k}-N$$ 
	   indeed constitutes an equilibrium to \eqref{1-1}.

	   \smallskip
	   
	   \subsection{Contour dynamics for single-layered patches}
	   The existence of such point vortex equilibrium inspires us to construct a series of steady patch solutions of \eqref{1-1}, which takes the form 
	   \begin{equation}\label{2-13}
	   	\omega(\boldsymbol x)=\frac{1}{\ep^2}\sum_{n=0}^{N-1}\boldsymbol 1_{\Omega_n^\ep}-\gamma
	   \end{equation}
	   with $\ep>0$ a small bifurcation parameter, $\boldsymbol 1_\Omega$ the characteristic function of domain $\Omega$, 
	   $$\Omega_n^\ep=\Omega_0^\ep+\frac{2\pi n}{N}\boldsymbol e_1, \quad n=0,1,\cdots,N-1$$ 
	   be $N$ small areas as perturbation of $B_\ep(\boldsymbol P_k)$, and $\gamma\in \mathbb{R}$ the strength of background vorticity around $N\pi$. 
	   
	   Using Biot-Savart law and Green-Stokes formula, the velocity of fluid can be recovered by
	   \begin{equation}\label{2-14}
	   	\boldsymbol{v}(\boldsymbol x)=\frac{1}{2\pi}\sum_{n=0}^{N-1}\int_{\partial \Omega_n^\ep}\left[\log\left(\frac{1}{|\boldsymbol x-\boldsymbol y|}\right)+H(\boldsymbol x,\boldsymbol y)\right]d\boldsymbol y+\frac{\gamma}{N}\sum_{n=0}^{N-1}\left(\boldsymbol x-\frac{2\pi n}{N}\boldsymbol e_1\right)^\perp.
	   \end{equation}
	   On the other hand, since $\omega(\boldsymbol x)$ has a patch structure, the conservation of momentum in Euler equation \eqref{1-1} can be transformed to 
	   \begin{equation}\label{2-15}
	   	\boldsymbol{v}(\boldsymbol x)\cdot \mathbf n(\boldsymbol x)=0, \ \ \ \forall \, \boldsymbol x \in \cup_{n=0}^{N-1}\partial \Omega^\varepsilon_n,
	   \end{equation}
	   where $\mathbf n(\boldsymbol x)$ is the normal vector to the patch boundary. We can assume that $\partial \Omega_0^\ep$ is parameterized by
	   \begin{equation*}
	   	(d,h)+\boldsymbol x(s) = (d+\varepsilon R(s)\cos(s), h+\varepsilon R(s)\sin(s)),
	   \end{equation*}
	   and hence
	   $$\mathbf n(\boldsymbol s)=(\ep R'(s)\sin(s)+\ep R(s)\cos(s),-\ep R'(s)\cos(s)+\ep R(s)\sin(s)).$$
	   If we combine \eqref{2-14}, \eqref{2-15}, and let
	   $$ \int\!\!\!\!\!\!\!\!\!\; {}-{} g(t)dt:=\frac{1}{2\pi}\int_0^{2\pi}g(t)dt$$
	   for simplicity, we can derive the contour dynamic equation for parametrization function $R(s)$ of $\Omega_0^\ep$  as follows
	   \begin{align}\label{2-16}
	   	0&= F(\varepsilon, \gamma, R(s))\\
	   	&=\frac{1}{2\ep R(s)}\int\!\!\!\!\!\!\!\!\!\; {}-{} \log\left(\frac{1}{ \left(R(s)-R(t)\right)^2+4R(s)R(t)\sin^2\left(\frac{s-t}{2}\right)}\right)\nonumber\\
	   	&\quad\times \left[(R(s)R(t)+R'(s)R'(t))\sin(s-t)+(R(s)R'(t)-R'(s)R(t))\cos(s-t)\right] dt\nonumber\\
	   	&+ \sum_{n=1}^{N-1} \frac{1}{ \ep R(s)} \int\!\!\!\!\!\!\!\!\!\; {}-{}\log\left( \frac{1}{\left| \boldsymbol x(s)-\boldsymbol x(t)-2\pi n\boldsymbol e_1/N\right|}\right)\left[(R(s)R(t)+R'(s)R'(t))\sin(s-t)\right.\nonumber\\
	   	&\quad \left. +(R(s)R'(t)-R'(s)R(t))\cos(s-t)\right] dt\nonumber\\
	   	&+\sum_{n=0}^{N-1} \frac{2\pi}{\ep R(s)} \int\!\!\!\!\!\!\!\!\!\; {}-{} H(\boldsymbol P_0+\boldsymbol x(s), \boldsymbol P_0+\boldsymbol x(t)+2\pi n\boldsymbol e_1/N)\left[(R(s)R(t)+R'(s)R'(t))\sin(s-t)\right.\nonumber\\
	   	&\quad \left. +(R(s)R'(t)-R'(s)R(t))\cos(s-t)\right] dt\nonumber\\
	   	&+\gamma\varepsilon R'(s)-\frac{(N-1)\gamma\pi}{N}\left(\frac{R'(s)\cos(s)}{R(s)}-\sin(s)\right)\nonumber\\
	   	&=F_1+F_2+F_3+F_4,\nonumber\\
	   \end{align}
	   where $F_1$ is induced by $\Omega_0^\ep$ itself, $F_2$ is the influence of the other patches, the regular part in the Green function $G$ contributes $F_3$, and $F_4$ is generated by the background vorticity.

	   \bigskip
		
	   \section{The functional analysis frame for single-layered patches}\label{sec3}
	   
	   Since $\Omega_n^\ep$ is a perturbation of the disk $B_\ep(\boldsymbol P_n)$, we assume that the boundary is radially parametrized by
	   $$R(s)=1+\ep u(s).$$
	   We use the following function spaces for $u:\mathbb S^1\to \mathbb R$ in the construction.
	   \begin{equation*}
	   	X^k=\left\{ u\in H^k(\mathbb S^1), \ f(s)= \sum\limits_{j=1}^{\infty}a_j\cos(js)\right\},
	   \end{equation*}
	   
	   \begin{equation*}
	   	Y^{k}=\left\{ u\in H^{k}(\mathbb S^1), \ u(s)= \sum\limits_{j=1}^{\infty}a_j\sin(js)\right\}.
	   \end{equation*}
	   The norms $\|\cdot\|_{X^k}$, $\|\cdot\|_{Y^k}$  in $X^k$ or $Y^{k}$ is the Sobolev $H^k$-norm on $\mathbb S^1:[0,2\pi)$ given by
	   \begin{equation*}
	   	\|u\|_{H^k(\mathbb S^1)}=\left(\sum_{j=1}^{\infty}|a_j|^2|j|^{2k}\right)^{\frac{1}{2}}.
	   \end{equation*}
	   For further usage of implicit function theorem in Section \ref{sec4}, we also denote
	   \begin{equation*}
	   	Y^{k}_0=\left\{ u\in H^{k}(\mathbb S^1), \ u(s)= \sum\limits_{j=2}^{\infty}a_j\sin(js)\right\}
	   \end{equation*}
	   as a subspace of $Y^k$ with $a_1=0$. Throughout this paper, we always assume that $k\geq 3$ and $\ep_0>0$ is a small constant.
	   
	   \subsection{Regularity of the contour dynamic equation}
	   
	   We will begin with the continuity of $F(\varepsilon, \gamma, R(x))$.
	   \begin{lemma}\label{lem3-1}
	   	Set $V:=\{ R(s)=1+\ep u(s)\mid u\in X^k,\|u\|_{X^k}<1\}$, then $F(\varepsilon, \gamma, R(s)): \left(-\varepsilon_0, \varepsilon_0\right)\times \mathbb{R} \times V \rightarrow Y^{k-1}$ is continuous.
	   \end{lemma}
	   
	   \begin{proof}
	   	Notice that $u(s)\in X^k$ is a even function, and $u'(s)$ is odd. It is easy to verify that $F_4$ is odd. By changing $t$ to $-t$ in the integrals of $F_1,F_2,F_3$, we can deduce that $F_1,F_2,F_3$ are odd too, and hence $F(\varepsilon, \gamma, R(s))$ is odd on $s$.
	   	
	   	To verify the regularity of $F(\varepsilon, \gamma, R(s))$, it is sufficient to consider the most singular term $F_1$. 
	   	Since $R(s)=1+\ep u(s)$, the possible singularity for $\ep=0$ may occur when we take the $\ep \rightarrow 0$ limit of $F_1$. To avoid the singularity, we will use the Taylor's formula
	   	\begin{equation*}
	   		\log\left(\frac{1}{A+B}\right)=\log\left(\frac{1}{A}\right)-\int_0^1\frac{B}{A+\tau B}d\tau,
	   	\end{equation*}
	   	where we let
	   	$$A:=4\sin^2\left(\frac{s-t}{2}\right),$$
	   	and 
	   	$$B:=\ep(u(s)-u(t))^2+\sin^2\left(\frac{s-t}{2}\right)(4u(s)+4u(t)+\ep u(s)u(t)).$$
	   	Then by using the fact that the sine function is odd, we have
	   	\begin{align*}
	   		F_1&(\ep,R(s))=\frac{1}{2\ep}\int\!\!\!\!\!\!\!\!\!\; {}-{} \log\left(\frac{1}{A+\ep B+O(\ep^2)}\right)\\
	   		&\quad \times\left[\big(1+\ep u (s)+\ep u(t)+O(\ep^2)\big)\sin(s-t)+\big(\ep u'(t)-\ep u'(s)+O(\ep^2)\big)\cos(s-t))\right]dt\\
	   		&=\frac{1}{2\ep}\int\!\!\!\!\!\!\!\!\!\; {}-{} \log\left(\frac{1}{A}\right)\sin(s-t)dt-\frac{1}{2}\int_0^1\int\!\!\!\!\!\!\!\!\!\; {}-{} \frac{B\sin(s-t)}{A+\tau \ep B+O(\ep^2)}dtd\tau\\
	   		&\quad+\frac{u(s)}{2}\int\!\!\!\!\!\!\!\!\!\; {}-{} \log\left(\frac{1}{A}\right)\sin(s-t)dt \\
	   		&\quad+\frac{1}{2}\int\!\!\!\!\!\!\!\!\!\; {}-{} \log\left(\frac{1}{A}\right)\left(u(t)\sin(s-t)+(u'(t)-u'(s))\cos(s-t)\right)dt+O(\ep)\\
	   		&=-\frac{1}{2}\int\!\!\!\!\!\!\!\!\!\; {}-{}\big(u(s)+u(t)\big)\sin(s-t)dt \\
	   		&\quad+\frac{1}{2}\int\!\!\!\!\!\!\!\!\!\; {}-{} \log\left(\frac{1}{4\sin^2\left( \frac{s-t}{2}\right) }\right)\left(u(t)\sin(s-t)+(u'(t)-u'(s))\cos(s-t)\right)dt+O(\ep)\\
	   		&=\frac{1}{2}\int\!\!\!\!\!\!\!\!\!\; {}-{} \log\left(\frac{1}{4\sin^2\left(\frac{t}{2}\right)}\right)\left(u(s-t)\sin(t)+(u'(s-t)-u'(s))\cos(t)\right)dt\\
	   		&\quad -\frac{1}{2}\int\!\!\!\!\!\!\!\!\!\; {}-{}u(t)\sin(s-t)dt+\ep\mathcal R_1(\ep, R(s)),
	   	\end{align*}  
	   	where we use
	   	$$\frac{B}{A}=u(s)+u(t)+O(\ep)$$
	   	for the second term in the third line, and $\mathcal R_1(\ep, R(s))$ is not singular with respect to $\ep$. Hence we see that $F_1\in L^2$ for all $(\varepsilon,R(s))\in \left(-\frac{1}{2}, \frac{1}{2}\right)\times V$.
	   	
	   	In the next step, we will show that $\partial^{k-1}F_1=\partial_x^{k-1}F_1(\varepsilon, R(s))\in L^2$.
	   	By splitting the term $R(s)R'(t)-R'(s)R(t)$ into $R(s)(R'(t)-R'(s))+R'(s)(R(s)-R(t))$ and taking the $(k-1)^{\mathrm{th}}$ derivatives of $F_1(\varepsilon, R(x))$, we see that $\partial^{k-1}F_1$ is equivalent to
	   	\begin{align*}
	   		&-\frac{\partial^{k-1}u(s)}{R(x)} \varepsilon F_1\\
	   		&+\frac{1}{2}\int\!\!\!\!\!\!\!\!\!\; {}-{}\log\left(\frac{1}{ \left(R(s)-R(t)\right)^2+4R(s)R(t)\sin^2\left(\frac{s-t}{2}\right)}\right)\\
	   		&\qquad\times \left((\partial^k u(s)  u'(t)+u'(s)\partial^k u(t))+(\partial^k u(t)-\partial^k u(s))\cos(s-t)\right)dt\\	
	   		& -\frac{1}{2}\int\!\!\!\!\!\!\!\!\!\; {}-{}\frac{(u(s)-u(t))(\partial^{k-1}u(s)-\partial^{k-1}f(y))+2(u(s)\partial^{k-1}u(t)+\partial^{k-1}u(s)u(t))\sin^2(\frac{s-t}{2})}{\left(R(s)-R(t)\right)^2+4R(s)R(t)\sin^2\left(\frac{s-t}{2}\right)}\\
	   		& \quad\times \left[(R(s)R(t)+R'(s)R'(t))\sin(s-t)+(R(s)R'(t)-R'(s)R(t))\cos(s-t)\right] dt\nonumber\\
	   		& -\frac{1}{2}\int\!\!\!\!\!\!\!\!\!\; {}-{}\frac{(u(s)-u(t))(u'(s)-u'(t))+2(u(s)u'(t)+u'(s)u(t))\sin^2(\frac{s-t}{2})}{\left(R(s)-R(t)\right)^2+4R(s)R(t)\sin^2\left(\frac{s-t}{2}\right)}\\
	   		& \quad\times \left[\varepsilon(\partial^{k-1}u(s)u'(t)+u'(s)\partial^{k-1}u(t))\sin(s-t)\right.\\
	   		&\qquad\left.+(R(s)\partial^{k-1}u(t)-\partial^{k-1}u(s)R(t))\cos(s-t)\right] dt.\nonumber
	   	\end{align*}
	   	Since $u\in H^k$ and $k\ge3$, we have $\|\partial^lu\|_{L^\infty}<C \|u\|_{H^k}$ for $l=0,1,2$. By  the mean value theorem and H\"older's inequality, we have $\partial^{k-1} F_1\in L^2$, which implies $F_1\in H^{k-1}$.  
	   	
	   	Then we need to prove that $F(\varepsilon, \gamma,R(s))$ is continuous in the norm of $Y^{k-1}$. We deal with the most singular term in $F_1$ as
	   	\begin{align*}
	   		F_1^*&=\frac{1}{2}\int\!\!\!\!\!\!\!\!\!\; {}-{} \log\left(\frac{1}{ \left(R(s)-R(t)\right)^2+4R(s)R(t)\sin^2\left(\frac{s-t}{2}\right)}\right)(u'(t)-u'(s))\cos(s-t) dt\\
	   		&=\frac{1}{2}\int\!\!\!\!\!\!\!\!\!\; {}-{} \log\left(\frac{1}{ \varepsilon^2\left(u(s)-u(t)\right)^2+4(1+\varepsilon u(s))(1+u(t))\sin^2\left(\frac{s-t}{2}\right)}\right)\\
	   		&\quad \times(u'(t)-u'(s))\cos(s-t) dt.
	   	\end{align*}
	   	The following notations will be used for convenience: for a general function $h$, denote
	   	\begin{equation}\label{3-1}
	   		\Delta g=g(s)-g(t), \ \ \ g=g(s), \ \ \ \tilde {g}=g(t),
	   	\end{equation}
	   	and
	   	\begin{equation}\label{3-2}
	   		D(g)=\varepsilon^{2}(\Delta g)^2+4(1+\varepsilon g)(1+\varepsilon\tilde {g})\sin^2\left(\frac{s-t}{2}\right).
	   	\end{equation}
	   	Then for $u_1, u_2\in V^k$, it holds
	   	\begin{align*}
	   		&\quad F_1^*(\varepsilon, u_1)-F_1^*(\varepsilon, u_2)\\
	   		&=-\frac{1}{2}\int\!\!\!\!\!\!\!\!\!\; {}-{} \log\left(\frac{1}{ D(u_1)}\right)\Delta u_1'\cos(s-t) dt+\frac{1}{2}\int\!\!\!\!\!\!\!\!\!\; {}-{} \log\left(\frac{1}{ D(u_2)}\right)\Delta u_2'\cos(s-t) dt\\
	   		&=-\frac{1}{2}\int\!\!\!\!\!\!\!\!\!\; {}-{} \log\left(\frac{1}{ D(u_1)}\right)(\Delta u_1'-\Delta u_2')\cos(s-t) dt\\
	   		&\quad+\frac{1}{2}\int\!\!\!\!\!\!\!\!\!\; {}-{} \left(\log\left(\frac{1}{ D(u_2)}\right)-\log\left(\frac{1}{ D(u_1)}\right)\right)\Delta u_2'\cos(s-t) dt\\
	   		&=I_1+I_2.
	   	\end{align*}
	   	It is easy to obtain that $\|I_1\|_{H^{k-1}}\le C\|u_1-u_2\|_{H^k}$. To deduce $\|I_2\|_{H^{k-1}}\le C\| u_1-u_2\|_{H^k}$, it is enough to prove $$\left|\left|\partial^{l} \left(\log\left(\frac{1}{ D( u_2)}\right)-\log\left(\frac{1}{ D(u_1)}\right)\right)\cdot|s-t|\right|\right|_{L^2}\leq C\| u_1- u_2\|_{H^k},$$
	   	for $l=0,\ldots, k-1$, which can be verified by direct calculations. Thus we have completed the proof for continuity.
	   \end{proof}
	   
	   For $(\ep,\gamma,R(s))\in (-\ep_0,\ep_0)\times \mathbb R\times V$ and $v(x)\in X^k$, let
	   \begin{equation*}
	   	\partial_uF(\ep,\gamma,R(s))v:=\lim_{\tau\to 0}\frac{1}{\tau}\big[F(\ep,\gamma,1+\ep(u+\tau v))-F(\ep,\gamma,1+\ep u)\big]
	   \end{equation*}
	   be the Gateaux derivative of $F(\ep,\gamma,R(s))$. By splitting
	   \begin{equation*}
	   	\left| \boldsymbol x(s)-\boldsymbol x(t)-2\pi n\boldsymbol e_1/N\right|^2=\left(\frac{2\pi n}{N}\right)^2+\ep A_n
	   \end{equation*}	
	   with
	   \begin{equation*}
	   	A_n=-\frac{4\pi n}{N}(R(x)\cos(s)-R(t)\cos(t))+\ep(R^2(s)+R(t)^2-2R(s)R(t)\cos(s-t)),
	   \end{equation*}
	   we can compute carefully to obtain
	   \begin{align}\label{3-3}
	   	&\quad \partial_uF (\varepsilon, \gamma, R(s)) v\\
	   	&=-\frac{ v(s)}{2R(s)^2}\int\!\!\!\!\!\!\!\!\!\; {}-{} \log\left(\frac{1}{ \left(R(s)-R(t)\right)^2+4R(s)R(t)\sin^2\left(\frac{s-t}{2}\right)}\right)\nonumber\\
	   	&\quad\times \left[(R(s)R(t)+R'(s)R'(t))\sin(s-t)+(R(s)R'(t)-R'(s)R(t))\cos(s-t)\right] dt\nonumber\\
	   	&+\frac{1}{2R(s)}\int\!\!\!\!\!\!\!\!\!\; {}-{}\log\left(\frac{1}{ \left(R(s)-R(t)\right)^2+4R(s)R(t)\sin^2\left(\frac{s-t}{2}\right)}\right)\nonumber\\
	   	&\quad\times \left[(v(s)R(t)+R(s)v(t)+v'(s)R'(t)+R'(s)h'(t))\sin(s-t)\right.\nonumber\\
	   	&\quad\quad+\left.(v(s)R'(t)+R(s)v'(t)-v'(s)R(t)-R'(s)v(t))\cos(s-t)\right] dt\nonumber\\
	   	&-\frac{1}{2R(s)}\int\!\!\!\!\!\!\!\!\!\; {}-{} \frac{(R(s)R(t)+R'(s)R'(t))\sin(s-t)+(R(s)R'(t)-R'(s)R(t))\cos(s-t)}{ \left(R(s)-R(t)\right)^2+4R(s)R(t)\sin^2\left(\frac{s-t}{2}\right)}\nonumber\\
	   	&\quad \times \left[2\left(R(s)-R(t)\right)(v(s)-v(t))+4(v(s)R(t)+R(s)v(t))\sin^2\left(\frac{s-t}{2}\right)\right] dt\nonumber\\
	   	&-\sum_{n=1}^{N-1} \frac{v(s)}{ R(s)^2} \int\!\!\!\!\!\!\!\!\!\; {}-{}\log\left( \frac{1}{\left|\boldsymbol x(s)-\boldsymbol x(t)-2\pi n\boldsymbol e_1/N\right|}\right)\nonumber\\
	   	&\quad \times\left[(R(s)R(t)+R'(s)R'(t))\sin(s-t)+(R(s)R'(t)-R'(s)R(t))\cos(s-t)\right] dt\nonumber\\
	   	&+\sum_{n=1}^{N-1} \frac{1}{ R(s)} \int\!\!\!\!\!\!\!\!\!\; {}-{}\log\left( \frac{1}{\left| \boldsymbol x(s)-\boldsymbol x(t)-2\pi n\boldsymbol e_1/N\right|}\right)\nonumber\\
	   	&\quad \times \left[(v(s)R(t)+R(s)v(t)+v'(s)R'(t)+R'(s)v'(t))\sin(s-t)\right.\nonumber\\
	   	&\quad\quad+\left.(v(s)R'(t)+R(s)v'(t)-v'(s)R(t)-R'(s)h(t))\cos(s-t)\right] dt\nonumber\\
	   	&-\sum_{n=1}^{N-1} \frac{1}{ 2R(s)} \int\!\!\!\!\!\!\!\!\!\; {}-{}\frac{\partial_{u}A_nv }{\left|\boldsymbol x(s)-\boldsymbol x(t)-2\pi n\boldsymbol e_1/N\right|^2}\nonumber\\
	   	&\quad \times \left[(R(s)R(t)+R'(s)R'(t))\sin(s-t)+(R(s)R'(t)-R'(s)R(t))\cos(s-t)\right] dt\nonumber\\
	   	&-\sum_{n=0}^{N-1} \frac{2\pi v(s)}{ R(t)^2} \int\!\!\!\!\!\!\!\!\!\; {}-{}H\left(\boldsymbol P_0+\boldsymbol{x}(s),\boldsymbol P_0+\boldsymbol{x}(t)+2\pi n\boldsymbol e_1/N\right)\nonumber\\
	   	&\quad \times\left[(R(s)R(t)+R'(s)R'(t))\sin(s-t)+(R(s)R'(t)-R'(s)R(t))\cos(s-t)\right] dt\nonumber\\
	   	&+\sum_{n=0}^{N-1} \frac{2\pi }{ R(s)} \int\!\!\!\!\!\!\!\!\!\; {}-{}H\left(\boldsymbol P_0+\boldsymbol{x}(s),\boldsymbol P_0+\boldsymbol{x}(t)+2\pi n\boldsymbol e_1/N\right)\nonumber\\
	   	&\quad \times \left[(v(s)R(t)+R(s)v(t)+v'(s)R'(t)+R'(s)v'(t))\sin(s-t)\right.\nonumber\\
	   	&\quad\quad+\left.(v(s)R'(t)+R(s)v'(t)-v'(s)R(t)-R'(s)v(t))\cos(s-t)\right] dt\nonumber\\
	   	&+\sum_{n=0}^{N-1} \frac{2\pi}{ R(s)} \int\!\!\!\!\!\!\!\!\!\; {}-{} \left[\nabla_1H\left(\boldsymbol P_0+\boldsymbol{x}(s),\boldsymbol P_0+\boldsymbol{y}(t)+2\pi n\boldsymbol e_1/N\right) \cdot(v(s)\cos (s), v(s)\sin(s))\right.\nonumber\\
	   	&\left.\quad+\nabla_2H\left(\boldsymbol P_0+\boldsymbol{x}(s),\boldsymbol P_0+\boldsymbol{x}(t)+2\pi n\boldsymbol e_1/N\right)\cdot(v(t)\cos (t),v(t)\sin (t))\right] \nonumber\\
	   	&\quad \times \left[(R(s)R(t)+R'(s)R'(t))\sin(s-t)+(R(s)R'(t)-R'(s)R(t))\cos(s-t)\right] dt,\nonumber\\
	   	&+\gamma\varepsilon v'(s)-\frac{(N-1)\gamma\pi}{N}\left(\frac{\varepsilon v'(s)\cos(s)}{R(s)}-\frac{\varepsilon R'(s)v'(s)\cos(s)}{R(s)^2}\right)\nonumber
	   \end{align}
	   where $\nabla_i H(\cdot,\cdot)$ means the gradient with respect to the $i^\mathrm{th}$ variable in $H$.

	   We have following lemma on the regularity of $\partial_uF(\ep,\gamma,R(s))v$. 
	   \begin{lemma}\label{lem3-2}
	   	The Gautex derivative
	   	$\partial_{u} F(\varepsilon, \gamma, R(s)): \left(-\varepsilon_0, \varepsilon_0\right)\times \mathbb{R} \times V \rightarrow L(X^{k}, Y^{k-1})$ is continuous.
	   \end{lemma}
	   
	   \begin{proof}
	   	For simplicity, we consider the most singular term in $F$, namely, the following special term in $F_1$:
	   	\begin{align*}
	   		F_1^*(\ep, u)=\frac{1}{2}\int\!\!\!\!\!\!\!\!\!\; {}-{} \log\left(\frac{1}{ \left(R(s)-R(t)\right)^2+4R(s)R(t)\sin^2\left(\frac{s-t}{2}\right)}\right)(u'(t)-u'(s))\cos(s-t) dt.\nonumber\\
	   	\end{align*}
	   	We need to prove
	   	\begin{equation}\label{3-4}
	   		\lim\limits_{\tau\to0}\left\|\frac{F_1^*(\varepsilon,u+\tau v)-F_1^*(\varepsilon, u_1)}{\tau}-\partial_uF_1^*(\varepsilon, u)v\right\|_{Y^{k-1}}= 0,
	   	\end{equation}
	   	where $\partial_uF_1^*(\varepsilon, u)v$ is
	   	\begin{align*}
	   		&\frac{1}{2}\int\!\!\!\!\!\!\!\!\!\; {}-{}\log\left(\frac{1}{ \left(R(s)-R(t)\right)^2+4R(x)R(y)\sin^2\left(\frac{s-t}{2}\right)}\right) (v'(t)-v'(s))\cos(s-t)dt\nonumber\\
	   		&-\frac{1}{2}\int\!\!\!\!\!\!\!\!\!\; {}-{} \frac{(u'(t)-u'(s))\cos(s-t)}{ \left(R(s)-R(t)\right)^2+4R(s)R(t)\sin^2\left(\frac{s-t}{2}\right)}\nonumber\\
	   		&\quad \times \left[2\varepsilon\left(R(s)-R(t)\right)(v(s)-v(t))+4\varepsilon(v(s)R(t)+R(s)v(t))\sin^2\left(\frac{s-t}{2}\right)\right] dt.
	   	\end{align*}
	   	Using the notations \eqref{3-1} and  \eqref{3-2} given in Lemma \ref{lem3-1}, we deduce
	   	\begin{equation*}
	   		\begin{split}
	   			&\frac{F_1^*(\varepsilon,u+\tau v)-F_1^*(\varepsilon, u_1)}{\tau}-\partial_uF_1^*(\varepsilon, u)(\varepsilon,u)v\\
	   			&=\frac{1}{2t}\int\!\!\!\!\!\!\!\!\!\; {}-{}(u'(s)-u'(t))\cos(s-t)\\
	   			&\quad \times\left(\log\left(\frac{1}{D(u+\tau v)}\right)-\log\left(\frac{1}{D(u)}\right)+\tau \frac{2\varepsilon^2\Delta u_1\Delta v+4\varepsilon(\tilde Rv+\tilde vR)\sin^2(\frac{s-t}{2})}{D(u)}\right)dt\\
	   			& +\frac{1}{2}\int\!\!\!\!\!\!\!\!\!\; {}-{}(v'(s)-v'(t))\cos(s-t)\left(\log\left(\frac{1}{D(u+\tau v)}\right)-\log\left(\frac{1}{D(u)}\right)\right)dt\\
	   			&=F_{11}^*+F_{12}^*.
	   		\end{split}
	   	\end{equation*}
	   	By taking the $(k-1)^{\mathrm{th}}$ partial derivatives of $F_{11}^*$, we find the most singular term is
	   	\begin{equation*}
	   		\begin{split}
	   			&\quad \frac{1}{2\tau}\int\!\!\!\!\!\!\!\!\!\; {}-{}\left(\log\left(\frac{1}{D(u+\tau v)}\right)-\log\left(\frac{1}{D(u)}\right)+\tau\frac{2\varepsilon^2\Delta u\Delta v+4\varepsilon(\tilde Rv+\tilde vR)\sin^2(\frac{s-t}{2})}{D(u)}\right)\\
	   			& \times(\partial^ku(s)-\partial^ku(t))\cos(s-t)dt.
	   		\end{split}
	   	\end{equation*}
	   	Using the  mean value theorem, we have
	   	\begin{equation*}
	   		\|F_{11}^*\|_{Y^{k-1}}\le C\tau^2\left\|\int\!\!\!\!\!\!\!\!\!\; {}-{}\partial^ku(x)-\partial^ku(y)dy\right\|_{L^2}\le C\tau^2\|u\|_{X^k}.
	   	\end{equation*}
	   	Taking the $(k-1)^{\mathrm{th}}$ partial derivatives of $F_{12}^*$, we have the most singular term,
	   	\begin{align*}
	   		\frac{1}{2}\int\!\!\!\!\!\!\!\!\!\; {}-{}(\partial^kv(s)-\partial^kv(t))\cos(s-t)\left(\log\left(\frac{1}{D(u+\tau v)}\right)-\log\left(\frac{1}{D(u)}\right)\right)dt.
	   	\end{align*}
	   	Using the  mean value theorem, we derive $$\left|\log\left(\frac{1}{D(u+\tau v)}\right)-\left(\frac{1}{D(u)}\right)\right|\leq C\tau,$$
	   	which implies
	   	\begin{equation*}
	   		\|F_{12}^*\|_{Y^{k-1}}\le C\tau\left\|\int\!\!\!\!\!\!\!\!\!\; {}-{}\partial^kv(s)-\partial^kv(t)dt\right\|_{L^2}\le C\tau\|v\|_{X^k}.
	   	\end{equation*}
	   	Letting $\tau\rightarrow 0$, we obtain \eqref{3-4}. As a consequence, we derive the existence of the Gateaux derivative of $F_1^*$. To prove the continuity of $\partial_u F_i, i=1,\cdots,4$, one just need to verify by the definition similarly.
	   \end{proof}

	   \subsection{Linearization at $F (0, \gamma,1)$}
	   
	   Letting $\varepsilon=0$ and $R(s)\equiv1$ in \eqref{3-3} as the expansion for $\partial_{u}F(\ep,\gamma,R(s))$ , we have
	   \begin{align}\label{3-5}
	   	&\quad\partial_uF (0, \gamma,1) v\\
	   	&=\frac{1}{2}\int\!\!\!\!\!\!\!\!\!\; {}-{}\log\left(\frac{1}{4\sin^2\left(\frac{t}{2}\right)}\right) \left[v(s-t)\sin(t)+(v'(s-t)-v'(s))\cos(t)\right] dt\nonumber\\
	   	&-\frac{1}{2}\int\!\!\!\!\!\!\!\!\!\; {}-{} \sin(t)v(s-t)dt.\nonumber
	   \end{align}
	   The following lemma claims the linearization at $F (0, \gamma,1)$ is an isomorphism from $X^k$ to $Y_0^{k-1}$.
	   \begin{lemma}\label{lem3-3}
	   	Let $v(s)=\sum_{j=1}^\infty a_j\cos(js)$ be in $X^k$, then it holds
	   	$$\partial_uF (0, \gamma,1)v=\sum_{j=1}^\infty \frac{j-1}{2}a_j\sin(js).$$
	   	Moreover, for each $\gamma\in \mathbb R$, $\partial_uF (0, \gamma,1): X^k\to Y_0^{k-1}$ is an isomorphism.
	   \end{lemma}
	   \begin{proof}
	   	We begin with the first term in \eqref{3-5}:
	   	\begin{align*}
	   		&\frac{1}{2}\int\!\!\!\!\!\!\!\!\!\; {}-{}\log\left(\frac{1}{4\sin^2\left(\frac{t}{2}\right)}\right) \left[v(s-t)\sin(t)+(v'(s-t)-v'(s))\cos(t)\right] dt\nonumber\\
	   		&=\sum_{j=1}^\infty \frac{a_j}{2}\int\!\!\!\!\!\!\!\!\!\; {}-{}\log\left(\frac{1}{4\sin^2\left(\frac{t}{2}\right)}\right) \left[\cos(j(s-t))\sin(t)-j\sin(j(s-t))\cos(t)+j\sin(js)\cos(t)\right] dt\nonumber\\
	   		&= \frac{a_1\sin(s)}{2}\int\!\!\!\!\!\!\!\!\!\; {}-{}\log\left(\frac{1}{4\sin^2\left(\frac{t}{2}\right)}\right) \left[\sin(t)\sin(t)-\cos(t)\cos(t)+\cos(t)\right] dt\nonumber\\
	   		&\quad+\sum_{j=2}^\infty \frac{a_j\sin(js)}{2}\int\!\!\!\!\!\!\!\!\!\; {}-{}\log\left(\frac{1}{4\sin^2\left(\frac{t}{2}\right)}\right) \left[\sin(jt)\sin(t)-j\cos(jt)\cos(t)+j\cos(t)\right] dt\nonumber\\
	   		&= \frac{a_1\sin(s)}{2}\int\!\!\!\!\!\!\!\!\!\; {}-{}\log\left(\frac{1}{4\sin^2\left(\frac{t}{2}\right)}\right) \left[-\cos(2t)+\cos(t)\right] dt\nonumber\\
	   		&\quad+\sum_{j=2}^\infty \frac{a_j\sin(js)}{2}\int\!\!\!\!\!\!\!\!\!\; {}-{}\log\left(\frac{1}{4\sin^2\left(\frac{t}{2}\right)}\right) \left[(1-j)\cos((j-1)t)-(j+1)\cos((j+1)t)+2j\cos(t)\right] dt\nonumber\\
	   		&= \frac{a_1\sin(s)}{2}\left(-\frac{1}{2}+1\right)+\sum_{j=2}^\infty \frac{a_j\sin(js)}{4}\left(-1-1+2j\right)\nonumber\\
	   		&= \frac{a_1\sin(s)}{4}+\sum_{j=2}^\infty \frac{j-1}{2}a_j\sin(js),\nonumber\\
	   	\end{align*}
	   	where we use the identity (Lemma 3.3 \cite{Cas1})
	   	\begin{align*}
	   		\int\!\!\!\!\!\!\!\!\!\; {}-{} \cos(mt)\log\left(\frac{1}{4\sin^2\left(\frac{t}{2}\right)}\right)d t&= \begin{cases} \frac{1}{|m|},  &m\in \mathbb Z, m\neq 0,\\  2\ln2, &m=0,\end{cases}\\ 
	   		\int\!\!\!\!\!\!\!\!\!\; {}-{} \sin(mt)\log\left(\frac{1}{4\sin^2\left(\frac{t}{2}\right)}\right)d t&= 0.
	   	\end{align*}
	   	By the trigonometric formula, the second term in \eqref{3-5} can be calculated as follows
	   	\begin{align*}
	   		-\frac{1}{2}\int\!\!\!\!\!\!\!\!\!\; {}-{} \sin(y)v(s-t)dt
	   		&=\sum_{j=1}^\infty -\frac{a_j}{2} \int\!\!\!\!\!\!\!\!\!\; {}-{} \sin(t)\cos(j(s-t))dt\nonumber\\
	   		&=\sum_{j=1}^\infty \frac{-a_j }{4}\int\!\!\!\!\!\!\!\!\!\; {}-{} (\sin(js+(1-j)t)-\sin((j+1)s-jt))dt\nonumber\\
	   		&= \frac{-a_1\sin(s)}{4}.\nonumber
	   	\end{align*}
	   	Thus we obtain
	   	\begin{equation*}
	   		\partial_uF (0, \gamma,1)v=\sum_{j=1}^\infty \frac{j-1}{2}a_j \sin(js).
	   	\end{equation*}
	   	
	   	Now we are going to verify that $\partial_uF (0, \gamma,1): X^k\to Y_0^{k-1}$ is an isomorphism. It is easy to see that $\partial_uF (0, \gamma,1)$ maps $X^k$ to $Y_0^{k-1}$. Thus we only need to show that for each $p(x)\in Y_0^{k-1}$ with the form $p(s)=\sum\limits_{j=2}^\infty b_j\sin(js)$, there exists an $v(s) \in X^k$ such that $\partial_uF (0, \gamma,1)v=p(s)$. From above calculation, we see that this $v$ can be given directly by
	   	\begin{equation*}
	   		v(s)=\sum\limits_{j=2}^\infty \frac{2}{j-1}b_j\cos(js)
	   	\end{equation*}
	   	and
	   	$$\|v\|^2_{X^k}=\sum\limits_{j=2}^\infty\frac{4j^{2k}}{(j-1)^2} b_j^2\le C\sum\limits_{j=2}^\infty b_j^2j^{2k-2}\le C\|p\|^2_{Y_0^{k-1}}.$$
	   	As a result, we deduce $v(s) \in X^k$ and complete the proof.
	   \end{proof}
	   
	   \bigskip
	   
	   \section{The existence for single-layered patches and other properties}\label{sec4}
	   
	   \subsection{The existence by implicit function theorem}
	   Using the Taylor's formula
	   \begin{equation*}
	   	\log\left(\frac{1}{A+B}\right)=\log\left(\frac{1}{A}\right)-\int_0^1\frac{B}{A+tB}dt,
	   \end{equation*}
	   we can expanded $F$ at $(0,\gamma,1)$ as
	   \begin{align}\label{4-1}
	   	F(\ep,\gamma,R(s))&=\left[\partial_uF (0, \gamma,1)u+\ep\mathcal R_1(\ep,R(s))\right]\\
	   	&\quad+\left[\sum_{n=1}^{N-1}\frac{N}{4\pi n}\sin(s)+\ep\mathcal R_2(\ep,R(s))\right]\nonumber\\
	   	&\quad+\left[\sum_{n=0}^{N-1}\partial_{11}H\left(\boldsymbol P_0,\boldsymbol P_0+\frac{2\pi n}{N}\boldsymbol e_1\right)\pi\sin(s)+\ep\mathcal R_3(\ep,R(s))\right]\nonumber\\
	   	&\quad+\left[-\frac{(N-1)\gamma\pi}{N}\sin(s)+\ep\mathcal R_4(\ep,\gamma,R(s))\right]\nonumber,
	   \end{align}
	   where $\ep\mathcal R_i, i=1,\cdots,4$ are left small terms in $F_1,\cdots,F_4$. It should be noticed that the condition
	   \begin{equation}\label{4-2}
	   	-\frac{(N-1)\gamma\pi}{N}+\sum_{n=1}^{N-1}\frac{N}{4\pi n}+\sum_{n=1}^{N-1}\partial_{11}H\left(\boldsymbol P_0,\boldsymbol P_0+\frac{2\pi n}{N}\boldsymbol e_1\right)\pi=0,
	   \end{equation}
	   namely, the coefficient of $\sin(s)$ in $F(0,\gamma,1)$ being $0$, is equivalent with $\gamma=N\pi$, which can be directly deduced by dynamic equations \eqref{2-11} and \eqref{2-12} as the model of $N$ point vortex system with strength $\pi$ on the polygon in $\mathcal D_{\boldsymbol \zeta}$.
	   
	   From the representation of the linearized operator $\partial_uF (0, \gamma,1)$ in Lemma \ref{lem3-3}, we see that it is only an isomorphism from $X^k$ to $Y_0^k$, where the first term $\sin(s)$ is lost. However, $F(\ep,\gamma,1+u(s))$ maps $X^k$ to $Y^k$. To apply the implicit function theorem, we should require the range of $F(\ep,\gamma,R(s))$ to be in $Y^k_0$, which means that the term $\sin(s)$ should not exist in the expansion \eqref{4-1}. To achieve this goal, we will adjust the strength of background vorticity as $\gamma=\gamma(\varepsilon,R(s))$ in the next lemma.
	   \begin{lemma}\label{lem4-1}
	   	There exists
	   	\begin{equation*}
	   		\gamma(\varepsilon, R(s)):=N\pi+\varepsilon\mathcal{R}_\gamma(\varepsilon,R(s))
	   	\end{equation*}
	   	with a continuous function $\mathcal{R}_\gamma(\varepsilon,R(s)):(-\ep_0,\ep_0)\times X^k\to \mathbb{R}$, such that $\tilde F(\varepsilon,R(s)):(-\ep_0,\ep_0)\times V\to Y_0^{k-1}$ is given by
	   	\begin{equation*}
	   		\tilde F(\varepsilon,R(s)):=F(\varepsilon,\gamma(\varepsilon,R(s)), R(s)).
	   	\end{equation*}
	   	Moreover, $\partial_u\mathcal{R}_\gamma(\varepsilon,R(s))v: X^k\to \mathbb{R}$ is continuous.
	   \end{lemma}
	   \begin{proof}
	   	It suffices to find $\gamma(\varepsilon, u):(-\ep_0,\ep_0)\times V\to\mathbb{R}$, such that the first Fourier coefficient of $\sin(s)$ vanishes in $F(\varepsilon,\gamma(\varepsilon,R(s)), R(s))$. From \eqref{4-1}, we deduce that $\gamma(\varepsilon, R(s))$ must satisfy
	   	\begin{align*}
	   		0&=\ep\tilde{\mathcal R}_1+\left[\sum_{n=1}^{N-1}\frac{N}{4\pi n}\sin(s)+\ep\tilde{\mathcal R}_2\right]\nonumber\\
	   		&\quad+\left[\sum_{n=0}^{N-1}\partial_{11}H\left(\boldsymbol P_0,\boldsymbol P_0+\frac{2\pi n}{N}\boldsymbol e_1\right)\pi\sin(s)+\ep\tilde{\mathcal R}_3\right]\nonumber\\
	   		&\quad+\left[-\frac{(N-1)\gamma(\ep, R(s))\pi}{N}\sin(s)+\ep\tilde{\mathcal R}_4\right],
	   	\end{align*}
	   	where
	   	$$\tilde{\mathcal{R}}_i=2\int\!\!\!\!\!\!\!\!\!\; {}-{}  \mathcal R_i\sin(s)ds, \quad i=1,\cdots,4$$
	   	is the contribution of $\mathcal R_i$ to the first Fourier coefficient. Notice that the coefficient for $\sin(s)$ in above equality is exactly the left hand side of \eqref{4-2} together with $\ep\tilde {\mathcal R}_1,\dots,\ep\tilde {\mathcal R}_4$ four $O(\ep)$-order perturbation terms. Using the fact $\gamma=N\pi$ solving \eqref{4-2}, we claim that
	   	\begin{equation*}
	   		\gamma(\varepsilon, R(s))=N\pi+\varepsilon\mathcal{R}_\gamma(\varepsilon,R(s)),
	   	\end{equation*}
	   	is an $O(\ep)$-perturbation near $N\pi$, where $\mathcal{R}_\gamma(\varepsilon,R(s)):X^k\to \mathbb{R}$ is some continuous function. Since $\partial_u{\mathcal{R}}_i(\varepsilon,R(s))$ $(i=1,\cdots,4)$ are continuous from Lemma \ref{lem3-2},  $\partial_u\mathcal{R}_\gamma(\varepsilon,R(s))v:X^k\to \mathbb{R}$ is also continuous. The proof is thus complete.	
	   \end{proof}
	   
	   Now, we are in a position to prove the existence of vortex patches.
	   
	   {\bf Proof of Theorem \ref{thm1}:}
	   We first prove that $\partial_u\tilde F(\varepsilon,R(s))v: X^k\to Y^{k-1}_0$ is an isomorphism. By chain rule, it holds
	   \begin{equation*}
	   	\partial_u\tilde F(0,1)v=\partial_\gamma F(0,N\pi,1)\partial_u\gamma(0,1)v+\partial_u F(0,N\pi,1)v.
	   \end{equation*}
	   From the discussion in Lemma \ref{lem4-1}, we see $\partial_u\gamma(0,1)=0$. Hence $\partial_u\tilde F(0,1)v=\partial_u F(0,N\pi,1)$, and we achieve the desired result by Lemma \ref{lem3-3}.
	   
	   According to the regularity of $F$ and Lemma \ref{lem4-1}, we can apply implicit function theorem, and claim that there exists $\epsilon_0>0$ such that
	   \begin{equation*}
	   	\left\{(\varepsilon,R)\in [-\epsilon_0,\epsilon_0]\times V \ : \ \tilde F(\varepsilon,R)=0\right\}
	   \end{equation*}
	   is parameterized by one-dimensional curve $\varepsilon\in [-\epsilon_0,\epsilon_0]\to (\varepsilon, R_\varepsilon)$. 
	   Moreover, we need to show that for each $\varepsilon\in[-\epsilon_0,\epsilon_0]\setminus \{0\}$, it always holds $R_\varepsilon\neq 1$. 
	   For this purpose, we can apply the Taylor's formula again on \eqref{4-1} to expand $\mathcal{R}_3$ and $\mathcal{R}_4$ as
	   \begin{equation*}
	   	\mathcal{R}_3+\mathcal{R}_4=\varepsilon \hat C\sin(2s)+o(\varepsilon)
	   \end{equation*}
	   with $\hat C$ being a constant depending on $d$ and $N$. On the other hand, by we have
	   $$F_1(\varepsilon,\gamma,1)=-\frac{(N-1)\gamma\pi}{N}\sin(s), \ \ \ \ \ \ \ F_2(\varepsilon,\gamma,1)=0.$$
	   Hence it follows that
	   \begin{equation*}
	   	F(\varepsilon,\gamma,1)\neq 0,  \ \ \ \ \ \forall \, \varepsilon\in [-\epsilon_0,\epsilon_0]\setminus \{0\}
	   \end{equation*}
	   as long as $\epsilon_0$ is chosen sufficiently small.
	   
	   Let $\tilde R(s)=R(-s)$. Our last step is to show that if $(\varepsilon,R)$ is a solution to $\tilde F(\varepsilon,R)=0$, then $(-\varepsilon, \tilde R)$ is also a solution. By replacing $t$ with $-t$ in the integral representation of $F(\varepsilon,\gamma,R)$ and using the fact that $R$ is an even function, we obtain $\gamma(\varepsilon,R)=\gamma(-\varepsilon,\tilde R)$. Then we can insert it into $F(\varepsilon,\gamma,R)$ and derive $\tilde F(-\varepsilon, R)=0$ by a similar substitution of variables as before.
	   
	   Hence we complete the proof of Theorem \ref{thm1}.\qed
	   
	   \smallskip
	   
	   \subsection{The regularity and convexity}
	   
	   In this section, we show the regularity of vortices boundary and the convexity of $\Omega_0^\varepsilon$. Recall that we split
	   \begin{equation*}
	   	F(\varepsilon, \gamma, R)=F_1+F_2+F_3+F_4
	   \end{equation*}
	   in \eqref{2-15}. We will take the $(k-1)^{\mathrm{th}}$ derivatives of $F=0$ and divide the equation into the form
	   \begin{equation}\label{4-3}
	   	L(\partial ^{k-1} u) +S(\partial^{k-1} u)=J(u),
	   \end{equation}
	   where $L$ is linear part of the most singular term and $L$ is invertible, $S$ has the same singularity as $L$ but $S$ is small and negligible, and $J$ is the remaining terms, which are more regular than $L$ and $S$. Then, we are able to improve the regularity of $u$ by the bootstrap method thanks to the difference of singularity on both sides of the equation \eqref{4-1} and the invertibility of $L+S$. 
	   
	   We choose $L$, $S$ and $J$ as follows.
	   \begin{equation}\label{4-4}
	   	L(\partial ^{k-1} u)=\int\!\!\!\!\!\!\!\!\!\; {}-{} \log\left(\frac{1}{\left(4\sin^2\left(\frac{t}{2}\right)\right)^{\frac{1}{2}}}\right)\cos(t)\left((\partial ^{k-1}u)'(s)-(\partial ^{k-1}u)'(s-t)\right)dt,
	   \end{equation}
	   
	   \begin{align}\label{4-5}
	   	S(\partial ^{k-1} u)
	   	&=-\gamma\varepsilon^{2}(\partial ^{k-1}u)'(s)+\frac{(N-1)\gamma\pi}{N}\cdot\frac{\varepsilon(\partial ^{k-1}u)'(s)\cos(s)}{R(s)}\nonumber\\
	   	&-\frac{\varepsilon u'(s)}{R(s)}\int\!\!\!\!\!\!\!\!\!\; {}-{} \log\left(\frac{1}{\left|\left(R(s)-R(s-t)\right)^2+4R(s)R(s-t)\sin^2\left(\frac{t}{2}\right)\right|^{\frac{1}{2}}}\right)\nonumber\\
	   	&\quad \times\sin(t)(\partial ^{k-1}u)'(s-t)dt\nonumber\\
	   	&-\frac{\varepsilon(\partial ^{k-1}u)'(s)}{ R(s)}\int\!\!\!\!\!\!\!\!\!\; {}-{} \log\left(\frac{1}{\left|\left(R(s)-R(s-t)\right)^2+4R(s)R(s-t)\sin^2\left(\frac{t}{2}\right)\right|^{\frac{1}{2}}}\right)\nonumber\\
	   	&\quad\times \sin(t)u'(s-t)+\cos(t)\left(u(s)-u(s-t)\right)dt\nonumber\\
	   	&-\ep\partial^{k-1}\mathcal R_3(\ep,R(s))-\ep\partial^{k-1}\mathcal R_4(\ep,R(s)) \nonumber\\
	   	&-\int\!\!\!\!\!\!\!\!\!\; {}-{} \cos(t)\left((\partial ^{k-1}f)'(s)-(\partial ^{k-1}f)'(s-t)\right)\nonumber\\
	   	&\quad\times \log\left[\frac{\left(4\sin^2\left(\frac{t}{2}\right)\right)^{\frac{1}{2}}}{\left|\left(R(s)-R(s-t)\right)^2+4R(s)R(s-t)\sin^2\left(\frac{t}{2}\right)\right|^{\frac{1}{2}}}\right]dt\nonumber\\
	   	&=S_1+S_2+S_3+S_4+S_5,
	   \end{align}
	   and  
	   \begin{align}\label{4-6}
	   	J(f)
	   	&=\partial ^{k-1}\bigg((\frac{1}{\ep R(x)}\int\!\!\!\!\!\!\!\!\!\; {}-{} \log\left(\frac{1}{ \left(R(s)-R(t)\right)^2+4R(s)R(t)\sin^2\left(\frac{s-t}{2}\right)}\right)\nonumber\\
	   	&\quad\times \left[(R(s)R(t)+R'(s)R'(t))\sin(s-t)+(R(s)R'(t)-R'(s)R(t))\cos(s-t)\right] \bigg))dt\nonumber\\
	   	&\quad-\frac{\varepsilon u'(s)}{R(s)}\int\!\!\!\!\!\!\!\!\!\; {}-{} \log\left(\frac{1}{\left|\left(R(s)-R(s-t)\right)^2+4R(s)R(s-t)\sin^2\left(\frac{t}{2}\right)\right|^{\frac{1}{2}}}\right)\nonumber\\
	   	&\quad \times\sin(t)(\partial ^{k-1}u)'(s-t)dt\nonumber\\
	   	&\quad-\frac{\varepsilon(\partial ^{k-1}u)'(s)}{ R(s)}\int\!\!\!\!\!\!\!\!\!\; {}-{} \log\left(\frac{1}{\left|\left(R(s)-R(s-t)\right)^2+4R(s)R(s-t)\sin^2\left(\frac{t}{2}\right)\right|^{\frac{1}{2}}}\right)\nonumber\\
	   	&\quad\times \sin(t)u'(s-t)+\cos(t)\left(u(s)-u(s-t)\right)dt\nonumber\\
	   	&+\partial^{k-1}\left(\sum_{n=1}^{N-1}\frac{N}{4\pi n}\sin(s)\right)+\partial^{k-1}\left(\sum_{n=1}^{N-1}\partial_{11}H(\boldsymbol P_0,\boldsymbol P_0+\frac{2\pi n}{N}\boldsymbol e_1)\pi\sin(s)\right)\nonumber\\
	   	&-\frac{(N-1)\gamma\pi}{N}\partial_{k-1}\left(\frac{R'(s)\cos(s)}{R(s)}-\sin(s)\right)+\frac{(N-1)\gamma\pi}{N}\cdot\frac{\varepsilon(\partial ^{k-1}u)'(s)\cos(s)}{R(s)}\nonumber\\
	   	&=J_1(f)+J_2(f)+J_3(f),
	   \end{align}
	   where $\ep\mathcal R_3$ and $\ep\mathcal R_4$ are small terms defined in \eqref{4-1}. 
	   \begin{lemma}\label{lem4-2}
	   	$L$ is linear and invertible and maps $H^{2}$ to $H^{1}$.
	   \end{lemma}
	   \begin{proof}
	   	It is obvious that $L$ is linear and maps $H^{2}$ to $H^{1}$. 
	   	Hence, we only need to prove that $L$ is invertible. For any $p(s)\in H^1(\mathbb S^1)$, without loss of generality, we assume
	   	\begin{equation*}
	   		p(s)=\sum_{j=1}^\infty \left(a_j\sin(js)+b_j\cos(js)\right).
	   	\end{equation*}
	   	Set 
	   	\begin{equation*}
	   		q(s)=\sum_{j=1}^\infty \left(\tilde{a}_j\sin(js)+\tilde{b}_j\cos(js)\right),
	   	\end{equation*}
	   	where $\tilde{a}_j$ and $\tilde{b}_j$ is given by
	   	\begin{equation*}
	   		\tilde{a}_j=\frac{2b_j}{j\left(\frac{1}{2j-2}+\frac{1}{2j+2}-1\right)},\quad
	   		\tilde{b}_j=\frac{-2a_j}{j\left(\frac{1}{2j-2}+\frac{1}{2j+2}-1\right)}, \quad  \ j\ge 2,
	   	\end{equation*}
	   	with $\tilde a_1=4b_1$ and $\tilde b_1=-4a_1$.
	   	It is easy to check that $0<|\bar{a}_j|, |\bar{b}_j|<\infty$ and hence $\tilde{a}_j, \tilde{b}_j$ are well-defined. Note that $|\tilde{a}_j|\sim b_j/j$ and $|\bar{b}_j|\sim a_j/j$ for $j$ large, thus $q\in H^{2}$. Thus we obtain $L(q)=p$, and the uniqueness of $q$ follows easily.
	   \end{proof}
	   \begin{lemma}\label{lem4-3}
	   	$||S(\partial ^{k-1} u)||_{H^1}\leq C(\varepsilon) ||\partial ^{k-1} u||_{H^{2}}$, where $C(\varepsilon) \to0$ as $\varepsilon\to 0$.
	   \end{lemma}
	   \begin{proof}
	   	It is obvious $||S_i||_{H^1}\leq C\varepsilon ||u||_{H^{2}}$ for $i=1,\cdots,4$.  For $S_5$, by noting that
	   	$$\left(R(s)-R(s-t)\right)^2+4R(s)R(s-t)\sin^2\left(\frac{t}{2}\right)=4\sin^2\left(\frac{t}{2}\right)+O(\ep),$$
	   	we deduce $||S_5||_{H^1}\leq C(\varepsilon) ||u||_{H^{2}}$ where $C(\varepsilon) \to0$ as $\varepsilon\to 0$.
	   \end{proof}
	   
	   \begin{lemma}\label{lem4-4}
	   	If $f\in X^{k}$, then $J(u)\in H^{1}$.
	   \end{lemma}
	   \begin{proof}
	   	Taking one more derivative to $J(f)$, we compute the most singular terms for each $J_i$ with $i=1,2,3$.\\
	   	1. The most singular terms for $\partial J_2$ are
	   	\begin{align*}
	   		\partial J_{21}&=\frac{\varepsilon u'(s)}{ R(s)}\int\!\!\!\!\!\!\!\!\!\; {}-{} \left(\frac{1}{\left| \left(R(s)-R(s-t)\right)^2+4R(s)R(s-t)\sin^2\left(\frac{t}{2}\right)\right|^{\frac{1}{2}}}\right)\\
	   		&\quad \times (\partial ^{k}u(s)-\partial ^{k}u(s-t))\cos(t) dt,
	   	\end{align*}
	   	\begin{align*}
	   		\partial J_{22}&=-\frac{\varepsilon u'(s)}{2 R(s)}\int\!\!\!\!\!\!\!\!\!\; {}-{}  \frac{R'(s-t)\sin(y)+[R(s)-R(s-t)]\cos(t)}{\left| \left(R(s)-R(s-t)\right)^2+4R(s)R(s-t)\sin^2\left(\frac{t}{2}\right)\right|}\\
	   		&\quad\times\Big( 2(R(s)-R(s-t))(\partial ^{k}u(s)-\partial ^{k}u(s-t)\Big.\\
	   		&\qquad\left.+4\left(\partial ^{k}u(s)R(s-t)+R(s)\partial ^{k}u(s-t)\right)\sin^2\left(\frac{t}{2}\right)\right)dt.
	   	\end{align*}
	   	The log term in $\partial J_{21}$ can be split into  $$\frac{1}{|\left(R(s)-R(s-t)\right)^2+4R(s)R(s-t)\sin^2\left(\frac{t}{2}\right)|^{\frac{1}{2}}}=\frac{1}{(R'(s)^2+R(s)^2)^{\frac{1}{2}}}\frac{1}{\left(4\sin^2\left(\frac{t}{2}\right)\right)^{\frac{1}{2}}}+\mathcal{L}_{21}(x,y),$$
	   	where $\mathcal{L}(s,t)$ is regular and belongs to $L^2$. Now $\partial J_{21}$ becomes
	   	\begin{align*}
	   		\partial J_{21}&=\frac{\varepsilon u'(s)}{R(s)}\int\!\!\!\!\!\!\!\!\!\; {}-{}  \log\left(\frac{1}{\left(4\sin^2\left(\frac{t}{2}\right)\right)^{\frac{1}{2}}}\right) (\partial ^{k}u(s)-\partial ^{k}u(s-t))dt\\
	   		&\quad+\frac{\varepsilon u'(s)}{R(s)}\int\!\!\!\!\!\!\!\!\!\; {}-{}  \log\left(\frac{1}{(R'(s)^2+R(s)^2)^{\frac{1}{2}}}\right) (\partial ^{k}u(s)-\partial ^{k}u(s-t))dt\\
	   		&\quad+\frac{\varepsilon u'(s)}{2\pi R(s)}\int_0^{2\pi} (\partial ^{k}u(s)-\partial ^{k}u(s-t))\mathcal{L}_{21}(s,t)\cos(t) dt.	
	   	\end{align*}
	   	Thus$||\partial J_{21}||_{L^2}\leq C||u||_{H^{k}}$.
	   	
	   	Similarly, by Taylor's expansion, we can divide the kernel in the first part of $\partial J_{22}$ into
	   	\begin{align*}
	   		&\quad \frac{(R'(s-t)\sin(y)+[R(s)-R(s-t)]\cos(y))(R(s)-R(s-t))}{\left| \left(R(s)-R(s-t)\right)^2+4R(s)R(s-t)\sin^2\left(\frac{t}{2}\right)\right|}\\
	   		&=\frac{\left(R'(s-t)+R'(s)\cos(t)\right)R'(s)}{R'(s)^2+R(s)^2}+{\mathcal{L}}_{22}(s,t),
	   	\end{align*}
	   	where $\mathcal{L}_{22}(s,t)$ is regular and belongs to $L^\infty$.  Thus, the first part of $\partial J_{22}$ belongs to $L^2$.
	   	As for the remaining part of $\partial J_{22}$, it is enough to notice that
	   	$$\left|\left|\frac{\left(R'(s-t)\sin(y)+[R(s)-R(s-t)]\cos(t)\right)\sin^2\left(\frac{t}{2}\right)}{\left| \left(R(s)-R(s-t)\right)^2+4R(s)R(s-t)\sin^2\left(\frac{t}{2}\right)\right|}\right|\right|_{L^\infty}\leq C. $$
	   	Thus, we conclude that $\partial J_{2}\in L^2$, which implies $J_2\in H^1$.\\
	   	2.  It is obvious that $ \partial J_3(u)\in L^2$ and hence $J_3\in H^1$.\\
	   	3. $$\partial J_1(f)=-\frac{(N-1)\gamma\pi}{N}\cdot\frac{\varepsilon^2u'(s)\partial ^ku(s)\cos(s)}{R^2(s)}+\mathrm{(lower \ order \ terms)}$$
	   	Obviously, $ \partial J_1(f)\in L^2$ and hence $J_1\in H^1$.
	   \end{proof}
	   
	   Combining the above lemmas, we obtain the following result of the regularity.
	   \begin{corollary}\label{coro4-5}
	   	If $u\in H^{k}$ solves $F(\varepsilon, \Omega, 1+\ep u(s))=0$ for some $\varepsilon$ small and $\gamma$, then $f\in H^{k+1}$ for any $k\geq 3$.
	   \end{corollary}
	   In view of the above result, by the bootstrap argument, we can eventually obtain the regularity of solution constructed in Section 2.1.
	   
	   \begin{theorem}\label{thm4-6}
	   	The solution $u$ to $F(\varepsilon, \gamma, 1+\ep u(s))=0$ is in $H^k$ for any $k\geq 3$ and hence $u(s)\in C^\infty$.
	   \end{theorem}
	   
	   We end  this section by showing the convexity of $\Omega_0^\varepsilon$ by calculating the curvature.
	   \begin{theorem}\label{thm4-7}
	   	For $\varepsilon$ sufficiently small, $R(s)=1+\varepsilon u(s)$ parametrizes convex patches.
	   \end{theorem}
	   \begin{proof}
	   	We will prove this fact by showing that the sign of the curvature $\boldsymbol \kappa$ at any given $s\in[0,2\pi)$ is positive. Indeed, by direct computation we have
	   	\begin{align*}
	   		\varepsilon \boldsymbol\kappa(s)=\frac{R(s)^2+2R'(s)^2-R(s)R''(s)}{\left(R(s)^2+R'(s)^2\right)^{\frac{3}{2}}}=\frac{1+O(\varepsilon)}{1+O(\varepsilon)}>0
	   	\end{align*}
	   	for $\varepsilon$ small, which implies the convexity.
	   \end{proof}
	   
	   From Theorem \ref{thm4-6} and Theorem \ref{thm4-7}, we can immediately derive Theorem \ref{thm2}.
	   
	   \bigskip
		
	   \section{Construction for the general case}\label{sec5}
	   
	   According to the discussion in Section \ref{sec2}, for fixed $(\kappa_1,\cdots,\kappa_N)\in\mathbb R^N$, if the location coordinates 
	   $$(\boldsymbol \nu_1^0,\cdots,\boldsymbol \nu_N^0)=\left(e^{\mathbf i(x_{11}^0+x_{12}^0\mathbf i)},\cdots e^{\mathbf i (x_{N1}^0+x_{N2}^0\mathbf i)}\right)$$
	   satisfies \eqref{2-10}, then there exists a general equilibrium state of point vortices to \eqref{1-1} of the form
	   \begin{equation*}
	   	\omega^*(\boldsymbol x)=\sum_{n=1}^N\kappa_n\boldsymbol \delta_{\boldsymbol x_n^0}-\sum_{n=1}^N\kappa_n
	   \end{equation*}
	   with $\boldsymbol x_n^0=(x_{n1}^0,x_{n2}^0)\in \mathbb{T}^2$. In this section, we  prove a general existence result, where the construction of patch solutions in \eqref{1-6} follows the same idea as in Theorem \ref{thm1}, but here we will adjust the center locations of these patches instead of the background vorticity. 
	   Note that in Theorem \ref{thm3}, besides a two-dimensional degenerate condition 
	   $$\mathrm{Rank}\,\left(\mathrm{Hess}(\mathcal W_N(\boldsymbol X^0))\right)=2N-2$$
	   is assumed on the Kirchhoff-Routh path function $\mathcal W_N$ at $\boldsymbol X^0=(\boldsymbol x_1^0,\cdots,\boldsymbol x_N^0)$, the location vector $\boldsymbol X^0$ is required to be centralized, that is to say, 
	   \begin{equation*}
	   	\sum\limits_{n=1}^Nx^0_{n1}=N\pi,  \quad\quad \sum\limits_{n=1}^Nx^0_{n2}=-\frac{N\log\rho}{2}.
	   \end{equation*}
	   These conditions will be used in Lemma \ref{lem5-2} to obtain the existence of the centralized location series $\mathbf X^\ep=(\boldsymbol x_1^\ep,\cdots,\boldsymbol x_N^\ep)$ as the center of $\Omega_1^\ep, \cdots,\Omega^\ep_N$ near $\boldsymbol{X^0}$.

	   To derive the contour dynamic equation for all $N$ patches, we let the $n^{\mathrm{th}}$ patch $\Omega_n^\ep$ be parameterized by 
	   $$(x_{n1}^\ep+\ep R_n(s)\cos(s), x_{n2}^\ep+\ep R_n(s)\sin(s))$$
	   with $R(s):\mathbb S^1\to\mathbb R$. However, since the shape of $\Omega_1^\ep, \cdots,\Omega^\ep_N$ can be different, we use the vector function
	   $$\boldsymbol R(s)=(R_1(s),\cdots,R_N(s)):\mathbb S^N \to\mathbb R^N$$
	   to denote all the boundary curves. By the 2D Euler equation \eqref{1-1}, the velocity field and the normal vector on patch boundary should satisfy
	   \begin{equation*}
	   	\boldsymbol{v}(\boldsymbol x)\cdot \mathbf n(\boldsymbol x)=0, \ \ \ \forall \, \boldsymbol x \in \cup_{n=1}^{N}\partial \Omega^\varepsilon_n,
	   \end{equation*}
	   from which we can use the Biot-Savart law to derive $N$ contour dynamic equations for $N$ patches:
	   \begin{align}\label{5-1}
	   	0&= F_m(\varepsilon, \mathbf X^\ep, \boldsymbol R(s))\\
	   	&=\frac{\kappa_m}{2\pi\ep R_m(s)}\int\!\!\!\!\!\!\!\!\!\; {}-{} \log\left(\frac{1}{ \left(R_m(s)-R_m(t)\right)^2+4R_m(s)R_m(t)\sin^2\left(\frac{s-t}{2}\right)}\right)\nonumber\\
	   	&\quad\times \left[(R_m(s)R_m(t)+R'_m(s)R'_m(t))\sin(s-t)+(R_m(s)R'_m(t)-R'_m(s)R_m(t))\cos(s-t)\right] dt\nonumber\\
	   	&+ \sum_{n=1,n\neq m}^{N} \frac{\kappa_n}{ \ep\pi R_m(s)} \int\!\!\!\!\!\!\!\!\!\; {}-{}\log\left( \frac{1}{\left| \boldsymbol x_m^\ep +\boldsymbol x(s)-\boldsymbol x_n^\ep-\boldsymbol x(t)\right|}\right)\left[(R_m(s)R_n(t)+R'_m(s)R'_n(t))\sin(s-t)\right.\nonumber\\
	   	&\quad \left. +(R_m(s)R'_n(t)-R'_m(s)R_n(t))\cos(s-t)\right] dt\nonumber\\
	   	&+\sum_{n=1}^{N} \frac{2\kappa_n}{\ep R_m(s)} \int\!\!\!\!\!\!\!\!\!\; {}-{} H(\boldsymbol x_m^\ep+\boldsymbol x(s),\boldsymbol x_n^\ep+\boldsymbol x(t))\left[(R_m(s)R_n(t)+R'_m(s)R'_n(t))\sin(s-t)\right.\nonumber\\
	   	&\quad \left. +(R_m(s)R'_n(t)-R'_m(s)R_n(t))\cos(s-t)\right] dt\nonumber\\
	   	&+\sum_{n=1}^{N}\kappa_n\varepsilon R'(s)-\left(\sum_{n=1,n\neq m}^N\kappa_n(x_{m1}^\ep-x_{n1}^\ep)\right)\left(\frac{R'(s)\cos(s)}{R(s)}-\sin(s)\right)\nonumber\\
	   	&\quad-\left(\sum_{n=1,n\neq m}^N\kappa_n(x_{m2}^\ep-x_{n2}^\ep)\right)\left(\frac{R'(s)\sin(s)}{R(s)}+\cos(s)\right), \quad\quad\quad m=1,\cdots,N\nonumber
	   \end{align}
	   for $N$ patches. Our task is to find a family of the function vector $\boldsymbol R(s)$ satisfying the integro-differential system \eqref{5-1}, which is also written down into the following vector form.
	   $$\boldsymbol F(\varepsilon, \mathbf X^\ep, \boldsymbol R(s))=\boldsymbol 0.$$
	   Similar to the single-layered case, we let
	   $$R_n(s)=1+\ep u_n(s).$$
	   But in this situation, due to the lack of symmetry in the $x_2$-direction, the Hilbert space for $u_n(s)$ is chosen to be
	   \begin{equation*}
	   	Z^{k}=\left\{ u\in H^{k}(\mathbb S^1), \ u(s)= \sum\limits_{j=1}^{\infty}\big[a_j\sin(js)+b_j\cos(js)\big]\right\},
	   \end{equation*}
	   and the norms $\|\cdot\|_{Z^k}$ for $u$ in $Z^k$ is the Sobolev $H^k$-norm on $\mathbb S^1:[0,2\pi)$ given by
	   \begin{equation*}
	   	\|u\|_{H^k(\mathbb S^1)}=\left(\sum_{j=1}^{\infty}\big[|a_j|^2+|b_j|^2\big]|j|^{2k}\right)^{\frac{1}{2}}.
	   \end{equation*}
	   For the application of implicit function theorem, we also let
	   \begin{equation*}
	   	Z^{k}_0=\left\{ u\in H^{k}(\mathbb S^1), \ u(s)= \sum\limits_{j=2}^{\infty}\big[a_j\sin(js)+b_j\cos(js)\big]\right\},
	   \end{equation*}
	   whose coefficients of first terms in Fourier series is absent. By denoting $\boldsymbol u=(u_1,\cdots,u_N)$ and
	   $$V:=\{ \boldsymbol R(s)=(1+\ep u_1,\cdots,1+\ep u_N) \mid \boldsymbol u\in (X^k)^N, \|u_m\|_{X^k}<1, m=1,\cdots, N\},$$ 
	   we can use a similar technique as in Lemma \ref{lem3-1} and Lemma \ref{lem3-2} to verify
	   \begin{itemize}
	   	\item[(i)] For $m=1,\cdots,N$, $F_m(\varepsilon, \boldsymbol X^\ep, \boldsymbol R(s)): \left(-\varepsilon_0, \varepsilon_0\right)\times (\mathbb{T}^2)^N \times V \rightarrow Z^{k-1}$ is continuous.
	   	\item[(ii)] For $1\le n,m\le N$, the Gateaux derivative
	   	$\partial_{u_n} F_m(\varepsilon, \boldsymbol X^\ep, \boldsymbol R(s)): \left(-\varepsilon_0, \varepsilon_0\right)\times (\mathbb{T}^2)^N \times V \rightarrow L(Z^{k}, Z^{k-1})$ is continuous.
	   \end{itemize}	
	   
	   Then we calculate the linearization at $(0, \boldsymbol X^\ep, (1,\cdots,1))$: For $v\in Z^k$, it holds
	   \begin{equation}\label{}
	   	\begin{split}
	   		&\partial_{u_m} F_m(0, \boldsymbol X^\ep, (1,\cdots,1))v_m\\
	   		&=\frac{1}{2}\int\!\!\!\!\!\!\!\!\!\; {}-{}\log\left(\frac{1}{4\sin^2\left(\frac{t}{2}\right)}\right) \left[v_m(s-t)\sin(y)+(v_m'(s-t)-v_m'(s))\cos(t)\right] dt\\
	   		&-\frac{1}{2}\int\!\!\!\!\!\!\!\!\!\; {}-{} \sin(t)v_m(s-t)dt,
	   	\end{split}
	   \end{equation}
	   and	$\partial_{u_n} F_m(0, \boldsymbol X^\ep, (1,\cdots,1))v_n=0$ for $n\neq m$. According to the calculation in Lemma \ref{lem3-3}, we have following analogy for Fourier representation of $\partial_{u_m} F_m(0, \boldsymbol X^\ep, (1,\cdots,1))$.
	   \begin{lemma}\label{lem5-1}
	   	Let $v_m(s)=\sum_{j=1}^\infty \big[a_j\sin(js)+b_j\cos(js)\big]$ be in $Z^k$, then it holds
	   	$$\partial_{u_m}F (0, \boldsymbol X^\ep,(1,\cdots,1))v_m=\sum_{j=1}^\infty \frac{j-1}{2}b_j\sin(js)-\sum_{j=1}^\infty \frac{j-1}{2}a_j\cos(js).$$
	   	Moreover, for each $\boldsymbol X^\ep\in (\mathbb{T}^2)^N$, $\partial_{u_m}F_m (0, \boldsymbol X^\ep,(1,\cdots,1)): Z^k\to Z_0^{k-1}$ is an isomorphism.
	   \end{lemma}
	   
	   Using Taylor's formula, we expand $F$ at $(0, \boldsymbol X^\ep,(1,\cdots,1))$ as
	   \begin{align}\label{5-3}
	   	&\quad F_m(\varepsilon, \boldsymbol X^\ep(\varepsilon,\boldsymbol R(s)), \boldsymbol R(s))=\partial_{u_m}F_m (0, \boldsymbol X^\ep,(1,\cdots,1))u_m\\
	   	&\quad+\left[-\sum_{n=1,n\neq m}^N\frac{\kappa_n}{2\pi}\frac{1}{|x_{m1}^\ep-x_{n1}^\ep|}+\sum_{n=1}^{N}\kappa_n\partial_{11}H(\boldsymbol x_m^\ep,\boldsymbol x_n^\ep)+\sum_{n=1,n\neq m}^N\kappa_n(x_{m1}^\ep-x_{n1}^\ep)\right]\sin(s)\nonumber\\
	   	&\quad+\left[-\sum_{n=1,n\neq m}^N\frac{\kappa_n}{2\pi}\frac{1}{|x_{m2}^\ep-x_{n2}^\ep|}+\sum_{n=1}^{N}\kappa_n\partial_{12}H(\boldsymbol x_m^\ep,\boldsymbol x_n^\ep)-\sum_{n=1,n\neq m}^N\kappa_n(x_{m1}^\ep-x_{n1}^\ep)\right]\cos(s)\nonumber\\
	   	&\quad+\ep\mathcal R(\ep,\boldsymbol X^\ep, \boldsymbol R(s))\nonumber
	   \end{align}
	   with $\ep \mathcal R(\ep,\boldsymbol X^\ep, \boldsymbol R(s))$ a small. Our final task is to find a family of $\boldsymbol X^\ep(\varepsilon, \boldsymbol R(s))$ such that image of $F_m(\varepsilon, \boldsymbol X^\ep(\varepsilon,\boldsymbol R(s)), \boldsymbol R(s))$ is in $Z_0^{k-1}$.
	   
	   \begin{lemma}\label{lem5-2}
	   	There exists
	   	\begin{equation*}
	   		\boldsymbol X^\ep(\varepsilon, \boldsymbol R(s)):=\boldsymbol X^0+\ep \tilde{\boldsymbol X}(\ep, \boldsymbol R(s))
	   	\end{equation*}
	   	with $\tilde{\boldsymbol X}(\ep, \boldsymbol R(s)):(-\ep_0,\ep_0)\times V\to (\mathbb{T}^2)^N $, such that for $m=1,\cdots,N$, $\tilde F_m(\varepsilon,\boldsymbol R(s)):(-\ep_0,\ep_0)\times V\to Z_0^{k-1}$ is given by
	   	\begin{equation*}
	   		\tilde F_m(\varepsilon,\boldsymbol R(s)):=F_m(\varepsilon, \boldsymbol X^\ep(\varepsilon,\boldsymbol R(s)), \boldsymbol R(s)).
	   	\end{equation*}
	   	Moreover, for $m=1,\cdots,N$, $\partial_{u_m}\tilde{\boldsymbol X}(\ep, \boldsymbol R(s)): Z^k\to (\mathbb{T}^2)^N$ is continuous.
	   \end{lemma}
	   \begin{proof}
	   	It suffices to find $\boldsymbol X^\ep(\ep, \boldsymbol R(s)):(-\ep_0,\ep_0)\times V\to (\mathbb{T}^2)^N$, such that the first Fourier coefficients of $\sin(s)$ and $\cos(s)$ vanishes in $F_m(\varepsilon, \boldsymbol X^\ep(\varepsilon,\boldsymbol R(s)), \boldsymbol R(s))$ for $m=1,\cdots, N$. By \eqref{5-3}, this condition is equivalent to the $2N$ algebraic equations
	   	\begin{equation}\label{5-4}
	   		\sum_{n=1,n\neq m}^N\frac{\kappa_n}{2\pi}\frac{1}{|x_{m1}^\ep-x_{n1}^\ep|}-\sum_{n=1}^{N}\kappa_n\partial_{11}H(\boldsymbol x_m^\ep,\boldsymbol x_n^\ep)-\sum_{n=1,n\neq m}^N\kappa_n(x_{m1}^\ep-x_{n1}^\ep)=2\ep\int\!\!\!\!\!\!\!\!\!\; {}-{} \mathcal R\sin(s)ds,
	   	\end{equation}
	   	and
	   	\begin{equation}\label{5-5}
	   		\sum_{n=1,n\neq m}^N\frac{\kappa_n}{2\pi}\frac{1}{|x_{m2}^\ep-x_{n2}^\ep|}-\sum_{n=1}^{N}\kappa_n\partial_{12}H(\boldsymbol x_m^\ep,\boldsymbol x_n^\ep)+\sum_{n=1,n\neq m}^N\kappa_n(x_{m1}^\ep-x_{n1}^\ep)=2\ep\int\!\!\!\!\!\!\!\!\!\; {}-{} \mathcal R\cos(s)ds,
	   	\end{equation}
	   	for $m=1,\cdots, N$ holding simultaneously. Since we already have
	   	\begin{equation*}
	   		\sum_{n=1,n\neq m}^N\frac{\kappa_n}{2\pi}\frac{1}{|x_{m1}^0-x_{n1}^0|}-\sum_{n=1}^{N}\kappa_n\partial_{11}H(\boldsymbol x_m^0,\boldsymbol x_n^0)-\sum_{n=1,n\neq m}^N\kappa_n(x_{m1}^0-x_{n1}^0)=0,
	   	\end{equation*}
	   	and
	   	\begin{equation*}
	   		\sum_{n=1,n\neq m}^N\frac{\kappa_n}{2\pi}\frac{1}{|x_{m2}^0-x_{n2}^0|}-\sum_{n=1}^{N}\kappa_n\partial_{12}H(\boldsymbol x_m^0,\boldsymbol x_n^0)+\sum_{n=1,n\neq m}^N\kappa_n(x_{m1}^0-x_{n1}^0)=0
	   	\end{equation*} 
	   	by $\omega^*$ in \eqref{1-5} being a steady point vortex system. It can be deduced immediately that \eqref{5-4}, \eqref{5-5} and the centralized condition
	   	\begin{equation*}
	   		\sum\limits_{n=1}^Nx^\ep_{n1}=N\pi,  \quad\quad \sum\limits_{n=1}^Nx^\ep_{n2}=-\frac{N\log\rho}{2}.
	   	\end{equation*}
	   	are solvable for 
	   	\begin{equation*}
	   		\boldsymbol X^\ep(\varepsilon, \boldsymbol R(s)):=\boldsymbol X^0+\ep \tilde{\boldsymbol X}(\ep, \boldsymbol R(s))
	   	\end{equation*}
	   	using the two-dimensional degenerate condition on the Kirchhoff-Routh path function $\mathcal W_N$ at $\boldsymbol X^0$. Then, by continuity of the Gateaux derivative
	   	$\partial_{u_n} F_m(\varepsilon, \boldsymbol X^\ep, \boldsymbol R(s))$, we can further verify that for $n=1,\cdots,N$, $\partial_{u_n}\tilde{\boldsymbol X}(\ep, \boldsymbol R(s)): Z^k\to (\mathbb{T}^2)^N$ is continuous.
	   \end{proof}

	   {\bf Proof of Theorem \ref{thm3}:} Let 
	   $$\tilde{\boldsymbol F}(\ep,\boldsymbol R(s))=(\tilde F_1(\varepsilon,\boldsymbol R(s)),\cdots,\tilde F_N(\varepsilon,\boldsymbol R(s))).$$
	   We first prove that $\partial_{\boldsymbol u}\tilde{\boldsymbol F}(\ep,\boldsymbol R(s))(v_1,\cdots,v_N): (Z^k)^N\to (Z^k_0)^N$ is an isomorphism. By chain rule, it holds
	   \begin{align*}
	   	&\quad\partial_{\boldsymbol u}\tilde F_m(\varepsilon,\boldsymbol R(s))(v_1,\cdots,v_N)\\
	   	&=\sum_{m=1}^N\nabla_{\boldsymbol x_m} F(0,\boldsymbol X^0,(1,\cdots,1))\partial_{u_m}\tilde{\boldsymbol X}(0, (1,\cdots,1))v_m+\partial_{u_m} F_m(0,\boldsymbol X^0,(1,\cdots,1))v_m.
	   \end{align*}
	   From the discussion in Lemma \ref{lem5-2}, we see $\tilde{\boldsymbol X}(0, (1,\cdots,1))v_m=\boldsymbol 0$. 
	   Hence, we have
	   $$\partial_{\boldsymbol u}\tilde F_m(\varepsilon,\boldsymbol R(s))(v_1,\cdots,v_N)=\partial_{u_m} F_m(0,\boldsymbol X^0,(1,\cdots,1))v_m.$$
	   Thus we have verified the isomorphism by Lemma \ref{lem5-1}.
	   
	   In view of the regularity of $\boldsymbol F$ and Lemma \ref{lem5-1}, we now can apply the implicit function theorem again, and claim that there exists $\epsilon_0>0$ such that
	   \begin{equation*}
	   	\left\{(\varepsilon,\boldsymbol R)\in [-\epsilon_0,\epsilon_0]\times V \ : \ \tilde {\boldsymbol F}(\varepsilon,\boldsymbol R(s))=\boldsymbol 0\right\}
	   \end{equation*}
	   is parameterized by a one-dimensional curve $\varepsilon\in [-\epsilon_0,\epsilon_0]\to (\varepsilon, \boldsymbol R_\varepsilon)$. 
	   
	   By a similar discussion as in the proof of Theorem \ref{thm1}, we see that for each $\varepsilon\in[-\epsilon_0,\epsilon_0]\setminus \{0\}$, it always holds $\boldsymbol R_\varepsilon\neq (1,\cdots,1)$, and use the symmetry to show that if $(\varepsilon,\boldsymbol R(s))$ is a solution to $\tilde {\boldsymbol F}(\varepsilon,\boldsymbol R(s))=0$, then $(-\varepsilon, \boldsymbol R(-s))$ is also a solution. The regularity and convexity can be derived by a same bootstrap procedure as for Theorem \ref{thm2}. Thus we complete the proof. \qed

	   \bigskip

	   \appendix
	   
	   \section{The $P$-function and its logarithmic derivative}\label{appA}
	   
	   The $P$-function defined on the annulus $\mathcal D_{\boldsymbol \zeta}=\{\boldsymbol\zeta\in\mathbb C \mid \rho<|\boldsymbol \zeta|\le 1\}$ is given by the infinite product
	   \begin{equation}\label{A-1}
	   P(\boldsymbol \zeta,\sqrt\rho)=(1-\boldsymbol \zeta)\prod_{n=1}^\infty(1-\rho^n\boldsymbol \zeta)(1-\rho^n/\boldsymbol \zeta),
	   \end{equation}
	   which has a simple zero at $\boldsymbol \zeta=1$ in $\mathcal D_{\boldsymbol \zeta}$, see also \cite{Abr,Cro}. The $K$-function defined in terms of the logarithmic derivative of $P(\boldsymbol \zeta,\sqrt\rho)$ is
	   \begin{equation*}
	   K(\boldsymbol \zeta,\sqrt\rho)=\frac{\boldsymbol \zeta P'(\boldsymbol \zeta,\sqrt\rho)}{P(\boldsymbol\zeta,\sqrt\rho)},
	   \end{equation*}
	   where the prime denotes the derivative with respect to $\boldsymbol \zeta$, namely, $P'(\boldsymbol \zeta,\sqrt\rho)=\frac{d P'(\boldsymbol \zeta,\sqrt\rho)}{d\boldsymbol\zeta}$. From \eqref{A-1}, we can deduce the infinite series formula
	   \begin{align*}
	   K(\boldsymbol \zeta,\sqrt\rho)&=\frac{\boldsymbol\zeta}{\boldsymbol\zeta-1}+\sum_{n=1}\left(\frac{-\rho^n\boldsymbol\zeta}{1-\rho^n\boldsymbol\zeta}+\frac{\rho^n\boldsymbol\zeta}{1-\rho^n/\boldsymbol\zeta}\right)\\
	   &=\frac{1}{\boldsymbol \zeta-1}+O(1), \quad \mathrm{as} \ \boldsymbol\zeta\to 1,
	   \end{align*} 
	   which has a simple pole singularity at $\boldsymbol\zeta=1$.
	   The properties of the $K$-function, 
	   \begin{equation}\label{A-2}
	   K(\rho\boldsymbol \zeta,\sqrt\rho)=K(\boldsymbol \zeta,\sqrt\rho)-1=-K(1/\boldsymbol\zeta,\sqrt\rho),
	   \end{equation}
	   \begin{equation}\label{A-3}
	   K(\boldsymbol \zeta,\sqrt\rho)+K(\bar{\boldsymbol \zeta},\sqrt\rho)=1, \quad \mathrm{on} \ |\boldsymbol \zeta|=1,
	   \end{equation}
	   \begin{equation}\label{A-4}
	   K(-1,\sqrt\rho)=\frac{1}{2},
	   \end{equation}
	   are applied in Section \ref{sec2} for finding the equilibrium of point vortices.
	   
	   \bigskip
	   
	   \noindent{\bf Conflict of interest statement:} On behalf of all authors, the corresponding author states that there is no conflict of interest.
	   
	   \bigskip
	   
	   \noindent{\bf Data available statement:} Our manuscript has no associated data.

	   \bigskip
	   
	   \noindent{\bf Acknowledgement:}  T.S. is partially supported by JSPS KAKENHI, Grant No. 23H00086. C.Z. is supported by by NNSF of China (Grant 12301142 and 12371212), CPSF (Grant 2022M722286), NSFSC (Grant 2024NSFSC1341), and International Visiting Program for Excellent Young Scholars of SCU.
	   
	   \bigskip	
	   
	   \phantom{s}
	   \thispagestyle{empty}

   \end{document}